%% file: uimp.tex
\documentclass[12pt, final]{amsart}
\usepackage{amsmath, amssymb}
\usepackage[author-year]{amsrefs}
\usepackage{graphicx, color}
\usepackage[utf8]{inputenc}

\newtheorem{theorem}{Theorem}[section]
\newtheorem*{theorem*}{Theorem}
\newtheorem{proposition}[theorem]{Proposition}
\newtheorem{lemma}[theorem]{Lemma}

\newtheorem{corollary}[theorem]{Corollary}

\theoremstyle{definition}
\newtheorem{definition}[theorem]{Definition}
\newtheorem*{definition*}{Definition}

\newtheorem{problem}[theorem]{Problem }

\theoremstyle{remark}
\newtheorem{remark}[theorem]{Remark}
\newtheorem{exercise}{Exercise}[section]

\newcommand{\binomial}[2]{\ensuremath{\left( \begin{matrix}#1 \\ #2 \end{matrix} \right)}}

\newcommand{\funcao}[5]{\ensuremath{
\begin{array}{rccl}
#1 : & #2 & \rightarrow & #3,\\ & #4 & \mapsto & #5
\end{array}}}

\newcommand{\indexar}[1]{}
\newcommand{\HdR}{\mathcal H_{\mathbf d}^{\mathbb R}}

\title{Newton iteration, conditioning and zero counting}
\author{Gregorio Malajovich}
\thanks{Lecture notes 
for the Santaló summer school on 
Recent Advances in Real Complexity and Computation, held
at the Palacio de la Magdalena, Santander, and sponsored
by the Universidad Internacional Menéndez Pelayo and
the Universidad de Cantábria}
\thanks{G.M. is partially supported by CNPq and CAPES (Brazil) and
by the MathAmSud grant {\em complexity}.}
\thanks{\copyright 2011 by Gregorio Malajovich applies to
Sections \ref{sec:Newton} to \ref{sec:EY}. Those appeared previously
in~\ocite{NONLINEAR-EQUATIONS}. \copyright 2012 by the author for
the remaining sections.}
\address{Departamento de Matemática Aplicada, Instituto de Matemática,
Universidade Federal do Rio de Janeiro.
Caixa Postal 68530, Rio de Janeiro RJ 21941-909, Brasil.
}
\email{gregorio.malajovich@gmail.com}
\urladdr{www.labma.ufrj.br/~gregorio}
	\date{July 13, 2012}
\begin{document}
\maketitle
\section{Introduction}

Mathematicians' obsession with counting led to many 
interesting and far-fetched problems. These lectures are
structured around a seemingly innocent counting
problem:

\begin{problem}[Real root counting]\label{pr:affine}
Given a system $\mathbf f=(f_1, \dots, f_n)$ of real polynomial equations 
in $n$ variables, count the number 
of real solutions.
\end{problem}

You can also find here a crash-course in Newton iteration. We will
state and analyze a Newton iteration based `inclusion-exclusion' 
algorithm to count (and find) roots of real polynomials. 
\medskip
\par

That algorithm was investigated in a sequence of three papers by
Felipe Cucker, Teresa Krick, Mario Wschebor and
myself \ycites{CKMW1,CKMW2,CKMW3}. Good numerical properties 
are proved in the first paper. For instance, the algorithm
is tolerant to controlled rounding error. Instead of
covering such technicalities, I will present a simplified version
and focus on the main ideas.

\medskip
\par

The interest of Problem~\ref{pr:affine} lies in the fact that it is
{\bf complete} for the complexity class $\#\mathbf P_{\mathbb R}$
over the {\bf BSS} (Blum-Shub-Smale) computation model over $\mathbb R$.
See \ocite{BCSS} for the BSS model of computation.
The class  $\#\mathbf P_{\mathbb R}$
was defined by Meer~\ycite{MEER2000} as the class
of all functions $f: \mathbb R^{\infty} \rightarrow \{0,1\}^{\infty} \cup \{\infty\}$
such that there exists a BSS machine $M$ working in polynomial time
and a polynomial $q$ satisfying
\[
f(\mathbf y) = \# \{ \mathbf z \in \,\mathbb R^{q(\mathrm{size}(\mathbf y)}: M(\mathbf y,\mathbf z) \text{ is an accepting
computation.}\}
\]
We refer to \ocite{BC06} for the proof of completeness and to \ocite{CKMW1} for
references on the subject of counting zeros. 
\medskip
\par
Counting real polynomial roots in $\mathbb R^n$ can be reduced to counting 
polynomial roots in $\mathbb S^{n+1}$. Given a degree $d$ polynomial $f(x_1, \dots, x_n)$,
its homogenization is $f^{\mathrm{homo}}(x_0, \dots, x_n) = 
x_0^d f(x_1/x_0, \dots, x_n/x_0)$. 

\begin{exercise}[Beware of infinity]
Find an homogeneous polynomial $g=g(\mathbf y,u)$ of degree 2 in $n+2$ variables such that
\[
\begin{split}
\# \{\mathbf x \in \mathbb R^n: f_1(\mathbf x) = \cdots = f_n(\mathbf x) = 0 \} +1 &= 
\\
=\frac{1}{2}
\# \{(\mathbf y,u) \in \mathbb S^{n+1} : f_1^{\mathrm{homo}}(\mathbf y) 
= &\cdots = f_n^{\mathrm{homo}}(\mathbf y) = g(\mathbf y,u) = 0 \} 
.
\end{split}
\]
\end{exercise}
%
%
Because of the exercise above, replacing $n$ by $n-1$,
Problem~\ref{pr:affine} reduces to:
\begin{problem}[Real root counting on $S^n$]\label{pr:sphere}
Given a system $\mathbf f=(f_1, \dots, f_n)$ of real homogeneous polynomial equations 
in $n+1$ variables, count the number 
of solutions in $S^n$.
\end{problem}

\medskip
\par
\par
This course is organized as follows. We start by a review of
{\bf alpha-theory}. This theory
originated with a couple of 
theorems proved by Steve Smale \ycite{Smale-PE} and
improved subsequently by several authors. 
It allows to guarantee (quantitatively) from the 
available data that Newton iterations will converge
quadratically to the solution of a system of equations.

Then I will speak about the inclusion-exclusion algorithm.
It uses crucially several results of alpha-theory.

The complexity of the inclusion-exclusion algorithm depends upon a condition number.
By endowing the input space with a probability distribution,
one can speak of the expected value of the condition number and
of the expected running time. 
The final section is a review of the complexity analysis performed
in~\ocite{CKMW2} and \ocite{CKMW3}.
\medskip
\par

A warning: these lectures are informal. The model of computation is 
{\bf cloud computing}. This means that we will allow for exponentially many 
parallel processors (essentially, BSS machines)
at no additional cost. Moreover, we will be informal in the sense that
we will assume that square roots and operator norms can be computed
exactly in finite time. While this does not happen in the BSS model, 
those
can be approximated and all our algorithms can be rewritten as rigorous
BSS algorithms at the cost of a harder complexity analysis
~\cite{CKMW1}.

\begin{exercise}
What would happen if you could design a true polynomial time algorithm to solve
Problem~\ref{pr:sphere}?
\end{exercise}

\paragraph{Acknowledgments} I would like to thank Teresa Krick,
Felipe Cucker and Mike Shub for pointing out some mistakes in a previous 
version.

\tableofcontents

\part[Newton iteration]{Newton Iteration and Alpha theory}
\label{CH:NEWTON}
\section{Outline}
\label{sec:Newton}

Let $\mathbf f$ be a 
mapping between Banach spaces. {\bf Newton Iteration} is defined by
\[
N( \mathbf f, \mathbf x ) = \mathbf x - D\mathbf f(\mathbf x)^{-1}
\mathbf f(\mathbf x)
\] 
wherever $D\mathbf f(\mathbf x)$ exists and is bounded. 
Its only possible
fixed points are those satisfying $\mathbf f(\mathbf x) = 0$.
When $\mathbf f(\mathbf x) = 0$ and $D\mathbf f(\mathbf x)$ is
invertible, we say that $\mathbf x$ is a {\bf nondegenerate zero}
of $\mathbf f$.

It is well-known that Newton iteration is quadratically
convergent in a neighborhood of a nondegenerate zero $\zeta$.
Indeed, $N(\mathbf f, \mathbf x) - \zeta = D^2\mathbf f(\zeta) 
(\mathbf x-\zeta)^2 + \cdots$.

There are two main approaches to quantify how fast is quadratic
convergence. One of them, pioneered by \ocite{Kantorovich}
assumes that the mapping $\mathbf f$ has a bounded second
derivative, and that this bound is known.\indexar{Kantorovich}

The other approach, developed by Smale~\ycites{Smale-analysis, Smale-PE} and described
here, assumes that the mapping $\mathbf f$ is analytic. 
Then we will be able to estimate a neighborhood of quadratic
convergence around a given zero (Theorem~\ref{th:gamma}) or
to certify an `approximate root' (Theorem~\ref{th:alpha})
from data that depends only
on the value and derivatives of $\mathbf f$ at one point.

A more general exposition on this subject may be found in
~\ocite{Dedieu-points-fixes}, covering also overdetermined 
and undetermined polynomial systems. 

\section{The gamma invariant}

Through this chapter, $\mathbb E$ and $\mathbb F$ are
Banach spaces, $\mathcal D \subseteq \mathbb E$ is open and 
$\mathbf f: \mathbb E \rightarrow \mathbb F$
is analytic.

This means that if $\mathbf x_0 \in \mathbb E$ is in the domain
of $\mathbb E$, then there is $\rho > 0$ with the property
that the series
\begin{equation}\label{alpha:taylor}
\mathbf f(\mathbf x_0) 
+
\mathbf Df(\mathbf x_0) (\mathbf x - \mathbf x_0)
+
\mathbf D^2f(\mathbf x_0) 
(\mathbf x - \mathbf x_0 , \mathbf x - \mathbf x_0)
+
\cdots
\end{equation}
converges uniformly for $\|\mathbf x - \mathbf x_0\|<\rho$, and its
limit is equal to $\mathbf f(\mathbf x)$ (For more details about
analytic functions between Banach spaces, see Nachbin~\ycites{Nachbin0,Nachbin}).

In order to abbreviate
notations, we will write \eqref{alpha:taylor} as
\[
\mathbf f(\mathbf x_0) 
+
\mathbf Df(\mathbf x_0) (\mathbf x - \mathbf x_0)
+
\sum_{k \ge 2} 
\frac{1}{k!}\mathbf D^kf(\mathbf x_0) 
(\mathbf x - \mathbf x_0)^k
\]
where the exponent $k$ means that $\mathbf x - \mathbf x_0$ appears $k$ times 
as an argument to the preceding multi-linear operator.

The maximum of such $\rho$ will be called the {\bf radius of convergence}.
(It is $\infty$ when the series \eqref{alpha:taylor} is globally convergent). This
terminology comes from univariate complex analysis. When 
$\mathbf E = \mathbb C$, the series will converge for all
$\mathbf x \in B(\mathbf x_0, \rho)$ and diverge for all
$\mathbf x \not \in \overline{B(\mathbf x_0, \rho)}$.
This is no more true in several complex variables, or Banach
spaces (Exercise~\ref{ex:alpha:series}).

The norm of a $k$-linear operator in Banach Spaces (such as the
$k$-th derivative) is the {\bf operator norm}, for instance
\[
\| D^k \mathbf f(\mathbf x_0) \|_{\mathbb E \rightarrow \mathbb F}
=
\sup_{\|\mathbf u_1\|_{\mathbb E} = \cdots = \|\mathbf u_k\|_{\mathbb E}=1}
\| D^k \mathbf f(\mathbf x_0) (\mathbf u_1, \dots, \mathbf u_k)\|_{\mathbb F}
.
\]

As long as there is no ambiguity, we drop the subscripts of the
norm.

\begin{definition}[Smale's $\gamma$ invariant]
\indexar{Smale's invariants!gamma}
\glossary{$\gamma(\mathbf f, \mathbf x)$&--&Invariant related to Newton iteration.}
Let $\mathbf f: \mathcal D \subseteq \mathbb E \rightarrow \mathbb F$ 
be an analytic mapping
between Banach spaces, and $\mathbf x_0 \in \mathcal D$.
When $D\mathbf f(\mathbf x_0)$ is invertible, define
\[
\gamma( \mathbf f, \mathbf x_0) 
=
\sup_{k \ge 2} 
\left(  
\frac{ \| D\mathbf f( \mathbf x_0)^{-1} D^k\mathbf f( \mathbf x_0) \| }
{k!}
\right)^{\frac{1}{k-1}}
.\]
Otherwise, set $\gamma(\mathbf f, \mathbf x_0) = \infty$.
\end{definition}

In the one variable setting, this can be compared to the radius of 
convergence $\rho$ of $\mathbf f'(\mathbf x)/\mathbf f'(\mathbf x_0)$, 
that satisfies
\[
\rho^{-1} = \limsup_{k \ge 2}
\left(  
\frac{ \| \mathbf f'( \mathbf x_0)^{-1} \mathbf f^{(k)}( \mathbf x_0) \| }
{k!}
\right)^{\frac{1}{k-1}}
.\]

More generally, 
\begin{proposition}\label{prop:gamma:analytic}
\indexar{analytic mapping!and the gamma invariant@and the $\gamma$ invariant}
Let $\mathbf f: \mathcal D \subseteq \mathbb E \rightarrow \mathbb F$ 
be a $C^{\infty}$ map between Banach spaces, and $\mathbf x_0 \in \mathcal D$. 
Then $f$ is
analytic in $x_0$ if and only if, $\gamma(f,x_0)$ is finite. 
The series
\begin{equation}
\label{eq:gamma:conv}
\mathbf f(\mathbf x_0) 
+
\mathbf Df(\mathbf x_0) (\mathbf x - \mathbf x_0)
+
\sum_{k \ge 2} 
\frac{1}{k!}\mathbf D^kf(\mathbf x_0) 
(\mathbf x - \mathbf x_0)^k
\end{equation}
is uniformly convergent for $\mathbf x \in B(\mathbf x_0, \rho)$ for
any $\rho < 1/\gamma(\mathbf f, \mathbf x_0))$.
\end{proposition}

\begin{proof}[Proof of the {\bf if} in Prop.\ref{prop:gamma:analytic}]
The series 
\[
\mathbf Df(\mathbf x_0)^{-1}
\mathbf f(\mathbf x_0) 
+
(\mathbf x - \mathbf x_0)
+
\sum_{k \ge 2} 
\frac{1}{k!}
\mathbf Df(\mathbf x_0)^{-1}
\mathbf D^kf(\mathbf x_0) 
(\mathbf x - \mathbf x_0)^k
\]
is uniformly convergent in $B(\mathbf x_0, \rho)$
where
\begin{eqnarray*}
\rho^{-1} &<& 
\limsup_{k \ge 2}
\left(  
\frac{ \| D\mathbf f( \mathbf x_0)^{-1} D^k\mathbf f( \mathbf x_0) \| }
{k!}
\right)^{\frac{1}{k}}
\\
&\le&
\limsup_{k \ge 2}
\gamma(\mathbf f, \mathbf x_0)^{\frac{k-1}{k}}
\\
&=&
\lim_{k \rightarrow \infty}
\gamma(\mathbf f, \mathbf x_0)^{\frac{k-1}{k}}
\\
&=&
\gamma(\mathbf f, \mathbf x_0)
\end{eqnarray*}
\end{proof}

\medskip

Before proving the {\bf only if} part of
Proposition~\ref{prop:gamma:analytic}, we need 
to relate the norm of a multi-linear map to the norm of the
corresponding polynomial.

\begin{lemma}\label{polarization:bound} Let $k \ge 2$.
\indexar{polarization bound}
Let $\mathbf T: \mathbb E^k \rightarrow \mathbb F$ be 
$k$-linear and symmetric.  
Let $\mathbf S: \mathbb E \rightarrow \mathbb F$,
$\mathbf S(\mathbf x)=T(\mathbf x,\mathbf x,\dots, \mathbf x)$ be the 
corresponding polynomial.
Then,
\[
\| \mathbf T \| \le 
e^{k-1} 
\sup_{\|\mathbf x\| \le 1} \|\mathbf S(\mathbf x)\|
\]
\end{lemma}

\begin{proof}
The polarization formula for (real or complex) tensors is
\[
\mathbf T(\mathbf x_1, \cdots, \mathbf x_k) = \frac{1}{2^kk!} 
\sum_{\substack{\epsilon_j=\pm 1\\
j=1, \dots, k}}
\epsilon_1
\cdots 
\epsilon_k
\mathbf S\left( \sum_{l=1}^k \epsilon_l \mathbf x_l \right)
\]
It is easily derived by expanding the expression inside 
parentheses.
There will be $2^k k!$ terms of the form
\[
\epsilon_1 \cdots \epsilon_k T(\mathbf x_1, \mathbf x_2, \cdots, \mathbf x_k)
\]
or its permutations.
All other terms miss at least one variable (say $\mathbf x_j$).
They cancel by summing for $\epsilon_j=\pm 1$.  

It follows that when $\| \mathbf x \|\le 1$,
\begin{eqnarray*}
\mathbf T(\mathbf x_1, \cdots, \mathbf x_k) 
&\le&
\frac{1}{k!}
\max_{\substack{\epsilon_j=\pm 1\\
j=1, \dots, k}}
\left\| \mathbf S\left( \sum_{l=1}^k \epsilon_l \mathbf x_l \right) \right\|
\\
&\le&
\frac{k^k}{k!}
\sup_{\|\mathbf x\| \le 1} \|\mathbf S(\mathbf x)\|
\end{eqnarray*}

The Lemma follows from using Stirling's formula,
\[k! \ge \sqrt{2 \pi k} k^k e^{-k} 
e^{1/(12k+1)}.\]
We obtain:
\[
\| \mathbf T \| \le 
\left(\frac{1}{\sqrt{2\pi k}} 
e^{-\frac{1}{12k+1}}\right)
e^{k} 
\sup_{\|\mathbf x\| \le 1} \|\mathbf S(\mathbf x)\|
.
\]

Then we use the fact that $k \ge 2$, hence $\sqrt{2 \pi k} \ge e$.
\end{proof}

\begin{proof}[Proof of Prop.\ref{prop:gamma:analytic}, {\bf only if} part]
Assume that the series~\eqref{eq:gamma:conv} converges 
uniformly for
$\|\mathbf x - \mathbf x_0\| < \rho$. Without loss of generality assume
that $\mathbb E=\mathbb F$ and $D\mathbf f(\mathbf x_0)=I$.  

We claim that
\[
\limsup_{k \ge 2} \sup_{\|\mathbf u\|=1} \|\frac{1}{k!} D^k\mathbf f(\mathbf x_0) \mathbf u^k\|^{1/k}
\le 
\rho^{-1}
.
\]

Indeed, assume that there is $\delta > 0$ and
infinitely many pairs $(k, \mathbf u)$ with
$\|\mathbf u_i\|=1$ and
\[
\|\frac{1}{k!} D^k\mathbf f(\mathbf x_0) \mathbf u^k\|^{1/k} > \rho^{-1} (1+\delta).
\] 

In that case, 
\[
\|\frac{1}{k!} D^k\mathbf f(\mathbf x_0) 
\left(\frac{\rho}{\sqrt{1+\delta}} \mathbf u\right)^k\| > 
\left(\sqrt{1+\delta}\right)^k
\] 
infinitely many times, and hence \eqref{eq:gamma:conv} does not
converge uniformly on $B(\mathbf x_0, \rho)$.

Now, we can apply Lemma~\ref{polarization:bound} to
obtain:
\begin{eqnarray*}
\limsup_{k \ge 2} 
\|\frac{1}{k!} D^k\mathbf f(\mathbf x_0)\|^{1/(k-1)}
&\le&
e \limsup_{k \ge 2} \sup_{\|\mathbf u\|=1} \|\frac{1}{k!} D^k\mathbf f(\mathbf x_0) \mathbf u^k\|^{\frac{1}{k-1}}
\\
&\le& 
e \lim_{k \rightarrow \infty}  \rho^{-(1+1/(k-1))}
\\
& =& e \rho^{-1}
\end{eqnarray*}

and therefore $\|\frac{1}{k!} D^kf(x_0)\|^{1/(k-1)}$ is bounded.
\end{proof}

\begin{exercise}
Show the polarization formula for Hermitian product:
\[
\langle \mathbf u, \mathbf v \rangle = 
\frac{1}{4} 
\sum_{\epsilon^4=1} \epsilon \|\mathbf u+ \epsilon \mathbf v\|^2
\]
Explain why this is different from the one in Lemma~\ref{polarization:bound}.
\end{exercise}

\begin{exercise} If one drops the uniform convergence hypothesis
in the definition of analytic functions, what happens to
Proposition~\ref{prop:gamma:analytic}?
\end{exercise}

\section{The $\gamma$-Theorems}

The following concept provides a good abstraction of
quadratic convergence.

\begin{definition}[Approximate zero of the first kind]
\indexar{approximate zero!of the first kind}
Let $\mathbf f: \mathcal D \subseteq \mathbf E \rightarrow \mathbf F$ be as above,
with $\mathbf f(\zeta)=0$. An {\bf approximate zero of the first
kind} associated to $\mathbf \zeta$ is a point $\mathbf x_0 \in
\mathcal D$, such that
\begin{enumerate}
\item The sequence $(\mathbf x)_{i}$ defined inductively by
$\mathbf x_{i+1} = N(\mathbf f, \mathbf x_i)$ is well-defined
(each $\mathbf x_i$ belongs to the domain of $\mathbf f$ and
$D\mathbf f(\mathbf x_i)$ is invertible and bounded).
\item 
\[
\| \mathbf x_i - \zeta \| \le 2^{-2^i+1} \| \mathbf x_0 - \zeta \|
.
\]
\end{enumerate}
\end{definition}

The existence of approximate zeros of the first kind is not
obvious, and requires a theorem.

\begin{theorem}[Smale]\label{th:gamma}
\indexar{approximate zero!of the first kind}
\indexar{theorem!gamma}
\indexar{theorem!Smale}
Let $\mathbf f: \mathcal D \subseteq 
\mathbb E \rightarrow \mathbb F$ be an analytic
map between Banach spaces. Let $\zeta$ be a nondegenerate
zero of $\mathbf f$. Assume that
\[
B=B\left(\zeta, \frac{3-\sqrt{7}}{2\gamma(\mathbf f, \zeta)}\right) \subseteq \mathcal D.
\]

Every $\mathbf x_0 \in B$ is an approximate zero of the first
kind associated to $\zeta$. The constant $(3-\sqrt{7})/2$ is the
smallest with that property.
\end{theorem}

Before going further, we remind the reader of the following fact.

\begin{lemma}\label{one:over:one:minus:t}
Let $d \ge 1$ be integer, and let $|t|<1$. Then,
\[
\frac{1}{(1-t)^d}
=
\sum_{k \ge 0} 
\binomial{k+d-1}{d-1}
t^k
.
\]
\end{lemma}

\begin{proof} Differentiate $d-1$ times the two sides of the expression
$1/(1-t)= 1+t+t^2+\cdots$, and then divide both sides by $d-1!$
\end{proof}

\begin{lemma} \label{psi}
The function $\psi(u) = 1 - 4u + 2 u^2$ is
\glossary{$\psi(u)$&--&The function $1-4u+2u^2$.}
decreasing and non-negative in $[0, 1-\sqrt{2}/2]$, and satisfies:
\begin{align}
\label{eq:psi:1}
\frac{u}{\psi(u)} &< 1 &&\text{ for $u \in [0, (5-\sqrt{17})/4)$ }\\
\label{eq:psi:2}
\frac{u}{\psi(u)} &\le \frac{1}{2} 
                        &&\text{ for $u \in [0, (3-\sqrt{7})/2]$ }.
\end{align}
\end{lemma}

The proof of Lemma~\ref{psi} is left to the reader (but see
Figure~\ref{fig:alpha:psi}).
\begin{figure}
\centerline{\resizebox{\textwidth}{!}{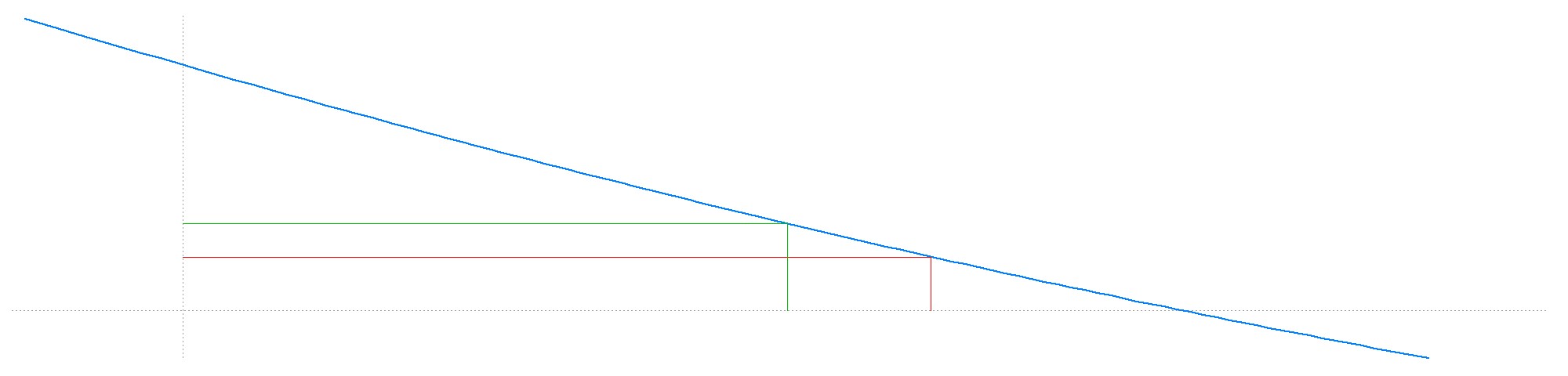}}
\caption{$y=\psi(u)$ \label{fig:alpha:psi}}
\end{figure}

Another useful result is:

\begin{lemma}\label{lem:neumann} 
Let $A$ be a $n \times n$ matrix.
Assume $\|A - I\|_2<1$.
Then $A$ has full rank and, for all $y$,
\[
\frac{\|y\|}{1+\|A-I\|_2}
\le
\| A^{-1} y\|_2 \le \frac{\|y\|}{1 - \| A - I\|_2}
.
\]
\end{lemma}

\begin{proof}
By hypothesis, $\|Ax\| > 0$ for all $x \ne 0$ so that $A$ has
full rank.
Let $y=Ax$. By triangular inequality,
\[
\|A x\| \ge \| x\| - \| (A-I)x \| \ge (1 -  \| (A-I) \|_2)\|x\|
.
\]
Also by triangular inequality,
\[
\|A x\| \le \| x\| + \| (A-I)x \| \le (1 +  \| (A-I) \|_2)\|x\|
.
\]
\end{proof}

The following Lemma will be needed:

\begin{lemma}\label{lem:gamma:derivative}
Assume that $u=\|\mathbf x - \mathbf y\| \gamma(\mathbf f,\mathbf x) < 
1-\sqrt{2}/{2}$. Then,
\[
\|
D\mathbf f(\mathbf y)^{-1} D\mathbf f(\mathbf x)
\| 
\le
\frac{(1-u)^2}{\psi(u)}
.
\] 
\end{lemma}

\begin{proof}
Expanding 
$\mathbf y \mapsto D\mathbf f(\mathbf x)^{-1}
D\mathbf f(\mathbf y)$ around
$\mathbf x$, we obtain:
\[
D\mathbf f(\mathbf x)^{-1}
D\mathbf f(\mathbf y) 
= I + \sum_{k \ge 2} \frac{1}{k-1!} 
D\mathbf f(\mathbf x)^{-1}
D^k\mathbf f(\mathbf x) (\mathbf y - \mathbf x)^{k-1}
.
\]

Rearranging terms and
taking norms, Lemma~\ref{one:over:one:minus:t} yields
\[
\| 
D\mathbf f(\mathbf x)^{-1}
D\mathbf f(\mathbf y) 
- I \| \le \frac{1}{(1-\gamma \|\mathbf y - \mathbf x\|)^2} - 1
.
\]

By Lemma~\ref{lem:neumann} we deduce that 
$D\mathbf f(\mathbf x)^{-1}D\mathbf f(\mathbf y)$ is invertible, and
\begin{equation}\label{gamma:leftmost}
\|
D\mathbf f(\mathbf y)^{-1} 
D\mathbf f(\mathbf x) 
\|
\le 
\frac{1}{1-\|
D\mathbf f(\mathbf x)^{-1} 
D\mathbf f(\mathbf y) 
-I\|}
=
\frac{(1-u)^2}{\psi(u)}
.
\end{equation}
\end{proof}

\medskip
\par

Here is the method for proving Theorem~\ref{th:gamma} and similar ones: first we study 
the convergence
of Newton iteration applied to a `universal' function. In this case, set
\[
h_{\gamma} (t) = t - \gamma t^2 - \gamma^2 t^3 - \cdots = 
t - \frac{\gamma t^2}{1-\gamma t}
.
\]
(See figure~\ref{fig:alpha:hgamma}).
\begin{figure}
\centerline{\resizebox{\textwidth}{!}{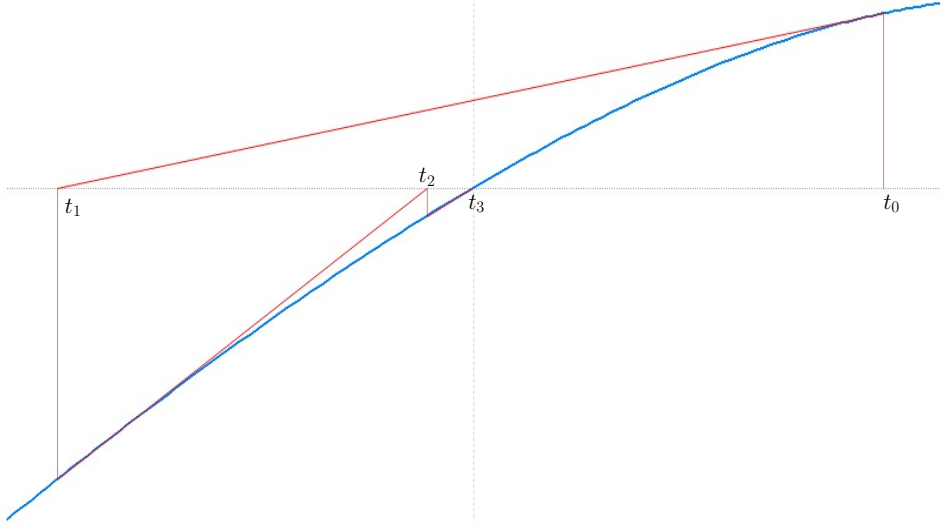}}
\caption{$y=h_{\gamma}(t)$ \label{fig:alpha:hgamma}}
\end{figure}

The function $h_{\gamma}$ has a zero at $t=0$, and $\gamma(h_{\gamma},0) = \gamma$.
Then, we compare the convergence of Newton iteration applied to an arbitrary 
function to the
convergence when applied to the universal function.

\begin{lemma}\label{lem:gamma:univ}
Assume that $0\le u_0 = \gamma t_0 < \frac{5 - \sqrt{17}}{4}$.
Then the sequences 
\[
t_{i+1} = N(h_{\gamma}, t_i) \text{ and } u_{i+1} = \frac{u_i^2}{\psi(u_i)}
\]
are well-defined for all $i$, $\lim_{i \rightarrow \infty} t_i = 0$, and
\[
\frac{|t_i|}{|t_0|} = \frac{u_i}{u_0} \le \left( \frac{u_0}{\psi(u_0)} \right)^{2^{i}-1}
.
\]

Moreover,
\[
\frac{|t_i|}{|t_0|} \le 2^{-2^i+1}
\]
for all $i$ if and only if $u_0 \le \frac{3 -\sqrt{7}}{2}$.
\end{lemma}

\begin{proof}
We just compute
\begin{eqnarray*}
h_{\gamma}'(t) &=& \frac{\psi(\gamma t)}{(1-\gamma t)^2} \\
t h_{\gamma}'(t) - h_{\gamma}(t) &=& - \frac{\gamma t^2}{(1-\gamma t)^2} \\
N(h_{\gamma}, t) &=& 
- \frac{\gamma t^2}{\psi(\gamma t)}. \\
\end{eqnarray*}

When $u_0 < \frac{5 - \sqrt{17}}{4}$, \eqref{eq:psi:1} implies that
the sequence $u_i$ is decreasing, and by induction
\[
u_i = \gamma |t_i| .
\]

Moreover, 
\[
\frac{u_{i+1}}{u_0} = 
\left(\frac{u_i}{u_0}\right)^2 \frac{u_0}{\psi(u_i)} 
\le
\left(\frac{u_i}{u_0}\right)^2 \frac{u_0}{\psi(u_0)} 
<
\left(\frac{u_i}{u_0}\right)^2 
.
\]

By induction,
\[
\frac{u_{i}}{u_0} \le 
\left( \frac{u_0}{\psi(u_0)} \right)^{2^i-1}
.
\]
This also implies that $\lim t_i = 0$. 

When furthermore $u_0 \le (3-\sqrt{7})/2$, 
$u_0 / \psi(u_0) \le 1/2$ by \eqref{eq:psi:2}
hence $u_i/u_0 \le 2^{-2^i+1}$. For the
converse, if $u_0 > (3-\sqrt{7})/2$,
then
\[
\frac{|t_1|}{|t_0|} = \frac{u_0}{\psi(u_0)} > \frac{1}{2}.
\]
\end{proof}

Before proceeding to the proof of Theorem~\ref{th:gamma},
a remark is in order.

Both Newton iteration and $\gamma$ are invariant with
respect to translation and to 
linear changes of coordinates: let $\mathbf g(\mathbf x) = A \mathbf f(\mathbf x-\zeta)$, where
$A$ is a continuous and invertible linear operator from
$\mathbb F$ to $\mathbb E$. Then
\[
N(\mathbf g, \mathbf x+\zeta) = N(\mathbf f,\mathbf x)+\zeta \text{ and } \gamma(\mathbf g, \mathbf x+\zeta) = \gamma(\mathbf f, \mathbf x)  
.
\]
Also, distances in $\mathbb E$ are invariant under translation.
\par

\begin{proof}[Proof of Th.\ref{th:gamma}]
Assume without loss of generality that
$\zeta=0$ and $D\mathbf f(\zeta) = I$. 
Set $\gamma = \gamma(\mathbf f,\mathbf x)$,
$u_0 = \| \mathbf x_0 \| \gamma$, and let
$h_{\gamma}$ and the sequence $(u_i)$ be as
in Lemma~\ref{lem:gamma:univ}.

We will bound
\begin{equation}\label{gamma:bound}
\| N(\mathbf f,\mathbf x) \| = \left\| \mathbf x - D\mathbf f(\mathbf x)^{-1} \mathbf f(\mathbf x) \right\|
\le 
\| D\mathbf f(\mathbf x)^{-1} \| 
\| \mathbf f(\mathbf x) - D\mathbf f(\mathbf x) \mathbf x \|
.
\end{equation}

The Taylor expansions of $\mathbf f$ and $D\mathbf f$ around $0$ are
respectively:
\[
\mathbf f(\mathbf x) = \mathbf x + \sum_{k \ge 2} \frac{1}{k!} D^k\mathbf f(0) \mathbf x^k 
\]
and
\begin{equation}\label{gamma:derivative}
D\mathbf f(\mathbf x) = I + \sum_{k \ge 2} \frac{1}{k-1!} D^k\mathbf f(0) \mathbf x^{k-1}
.
\end{equation}

Combining the two equations, above, we obtain:
\[
\mathbf f(\mathbf x) - D\mathbf f(\mathbf x) \mathbf x = \sum_{k \ge 2} \frac{k-1}{k!} D^k\mathbf f(0) \mathbf x^{k}.
\]

Using Lemma~\ref{one:over:one:minus:t} with $d=2$, the rightmost term in 
\eqref{gamma:bound} is bounded above by
\begin{equation}\label{gamma:rightmost}
\| \mathbf f(\mathbf x) - D\mathbf f(\mathbf x) \mathbf x \|
\le
\sum_{k \ge 2} (k-1) \gamma^{k-1} \|\mathbf x\|^k
=
\frac{ \gamma \|\mathbf x\|^2}{(1- \gamma \|\mathbf x\|)^2} 
.
\end{equation}

Combining Lemma~\ref{lem:gamma:derivative} and \eqref{gamma:rightmost} in \eqref{gamma:bound},
we deduce that
\[
\| N(\mathbf f,\mathbf x) \| \le \frac{ \gamma \|\mathbf x\|^2}{\psi(\gamma \|\mathbf x\|)}
.
\]

By induction, $u_i \le \gamma \|\mathbf x_i\|$.
When $u_0 \le (3-\sqrt{7})/2$, we obtain as in
Lemma~\ref{lem:gamma:univ} that

\[
\frac{ \|\mathbf x_i \| }{\| \mathbf x_0 \|} \le \frac{ u_i }{u_0} \le 2^{-2^i+1}.
\]

We have seen in Lemma~\ref{lem:gamma:univ} 
that the bound above fails for $i=1$ when $u_0 > (3-\sqrt{7})/2$.
\end{proof}

Notice that in the proof above,
\[
\lim_{i \rightarrow \infty} \frac{u_0}{\psi(u_i)} = u_0
.
\]

Therefore, convergence is actually faster than predicted by the
definition of approximate zero.
We proved actually a sharper result:

\begin{theorem}\label{th:gamma:sharp}
\indexar{theorem!gamma!sharp}
Let $\mathbf f: \mathcal D \subseteq 
\mathbb E \rightarrow \mathbb F$ be an analytic
map between Banach spaces. Let $\zeta$ be a nondegenerate
zero of $\mathbf f$. 
Let $u_0 < (5 - \sqrt{17})/4$.

Assume that
\[
B=B\left(\zeta, \frac{u_0}{\gamma(\mathbf f, \zeta)}\right) \subseteq \mathcal D.
\]

If $\mathbf x_0 \in B$, 
then the sequences
\[
\mathbf x_{i+1} = N(\mathbf f, \mathbf x_i) 
\text{ and }
u_{i+1} = \frac { u_i^2 }{\psi (u_i)} 
\]
are well-defined for all $i$, and 
\[
\frac { \left\| \mathbf x_i - \zeta \right\| }{ \left\| \mathbf x_0 - \zeta \right\| }
\le
\frac{u_i}{u_0}
\le
\left( \frac{u_0}{\psi(u_0)} \right)^{-2^i + 1}
.
\]
\end{theorem}

Table~\ref{table:alpha:gamma-conv} and Figure~\ref{graph:alpha:gamma-conv} show how
fast $u_i/u_0$ decreases in terms of $u_0$ and $i$.

\begin{table}
\centerline{
\begin{tabular}{|c|r@{.}lr@{.}lr@{.}lr@{.}lr@{.}l|}
\hline&\multicolumn{2}{c}{$1/32$}&\multicolumn{2}{c}{$1/16$}&\multicolumn{2}{c}{$1/10$}&\multicolumn{2}{c}{$1/8$}&\multicolumn{2}{c|}{$\frac{3-\sqrt{7}}{2}$}\\\hline
$ 1$&  4&810&  3&599&  2&632&  2&870&  1&000\\
$ 2$& 14&614& 11&169&  8&491&  6&997&  3&900\\
$ 3$& 34&229& 26&339& 20&302& 16&988& 10&229\\
$ 4$& 73&458& 56&679& 43&926& 36&977& 22&954\\
$ 5$&151&917&117&358& 91&175& 76&954& 48&406\\
\hline\end{tabular}
}
\caption{Values of $-log_2 (u_i/u_0)$ in function of $u_0$ and $i$. \label{table:alpha:gamma-conv}} 
\end{table}

\begin{figure}
\centerline{\resizebox{\textwidth}{!}{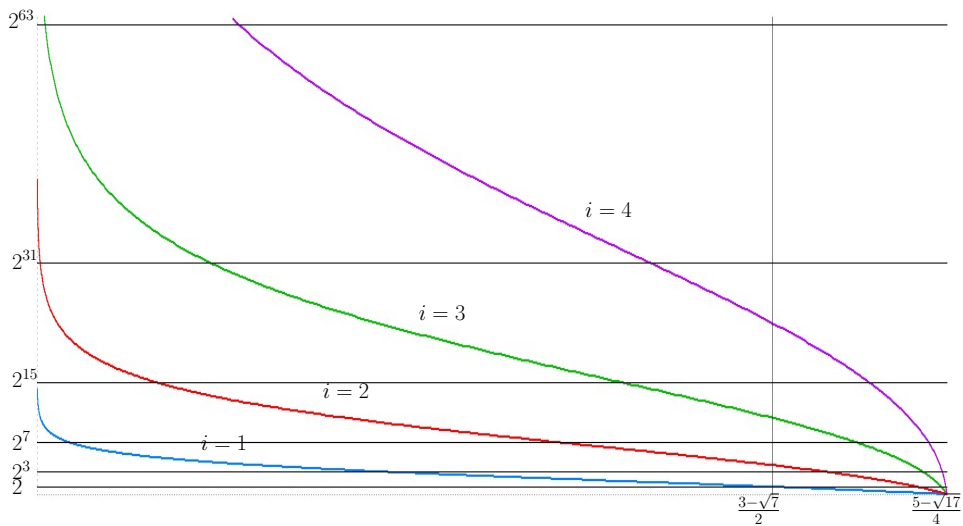}}
\caption{Values of $log_2 (u_i/u_0)$ in function of $u_0$ for
$i=1, \dots, 4$. \label{graph:alpha:gamma-conv}} 
\end{figure}
\medskip

To conclude this section, we need to address an important issue
for numerical computations. Whenever dealing with digital computers,
it is convenient to perform calculations in floating point format.
This means that each real number is stored as a {\bf mantissa} 
(an integer, typically no more than $2^{24}$ or $2^{53}$) times an
exponent. (The IEEE-754 standard for computer arithmetic 
~\cite{IEEE-754}
is taught at
elementary numerical analysis courses, see for instance~\ocite{Higham}*{Ch.2}).

By using floating point numbers, a huge gain of speed is obtained
with regard to
exact representation of, say, algebraic numbers. However,
computations are inexact (by a typical factor of $2^{-24}$
or $2^{-53}$). Therefore, we need to consider {\bf inexact}
Newton iteration. An obvious modification of the proof of 
Theorem~\ref{th:gamma} gives
us the following statement:

\begin{theorem}\label{th:gamma:robust}
\indexar{theorem!gamma!robust}
Let $\mathbf f: \mathcal D \subseteq 
\mathbb E \rightarrow \mathbb F$ be an analytic
map between Banach spaces. Let $\zeta$ be a nondegenerate
zero of $\mathbf f$. Let 
\[
0 \le 2 \delta \le u_0 \le 2 - \frac{\sqrt{14}}{2}\simeq 0.129\cdots
\]
Assume that
\begin{enumerate}
\item
\[
B=B\left(\zeta, \frac{u_0}{\gamma(\mathbf f, \zeta)}\right) \subseteq \mathcal D.
\]
\item $\mathbf x_0 \in B$, and the sequence $\mathbf x_i$ satisfies
\[
\| \mathbf x_{i+1} - N(\mathbf f, \mathbf x_i) \| \gamma(\mathbf f, \zeta) 
\le
\delta
\]
\item The sequence $u_i$ is defined inductively by
\[
u_{i+1} = \frac { u_i^2 }{\psi (u_i)} + \delta
.
\]
\end{enumerate}

Then the sequences $u_i$ and $\mathbf x_i$ are well-defined for all $i$, 
$\mathbf x_i \in \mathcal D$, and 
\[
\frac { \left\| \mathbf x_i - \zeta \right\| }{ \left\| \mathbf x_0 - \zeta \right\| }
\le
\frac{u_i}{u_0}
\le
\max
\left( 2^{-2^i+1}, 2 \frac{\delta}{u_0} \right)
.
\]
\end{theorem}

\begin{proof}
By hypothesis,
\[
\frac{u_0}{\psi(u_0)} + \frac{\delta}{u_0} < 1
\]
so the sequence $u_i$ is decreasing and positive.
For short, let $q=\frac{u_0}{\psi(u_0)} \le 1/4.$ 
By induction,
\[
\frac{u_{i+1}}{u_0} 
\le \frac{u_0}{\psi(u_i)} \left(\frac{u_i}{u_0}\right)^2 + \frac{\delta}{u_0}
\le \frac{1}{4} \left(\frac{u_i}{u_0}\right)^2 + \frac{\delta}{u_0}
.
\]

Assume that $u_i/u_0 \le 2^{-2^i+1}$. In that case,
\[
\frac{u_{i+1}}{u_0} 
\le
2^{-2^{i+1}} + \frac{\delta}{u_0}
\le
\max
\left(
2^{-2^{i+1}+1} , 2 \frac{\delta}{u_0}
\right)
.
\]

Assume now that $2^{-2^i+1}, u_i/u_0 \le 2 \delta/u_0$. In that case,
\[
\frac{u_{i+1}}{u_0} 
\le
\frac{\delta}{u_0}
\left(
\frac{\delta}{4u_0} 
+
1
\right)
\le
\frac{2\delta}{u_0}
=
\max
\left(
2^{-2^{i+1}+1} , 2 \frac{\delta}{u_0}
\right)
.
\]

From now on we use the assumptions, notations and estimates of the proof of 
Theorem~\ref{th:gamma}.
Combining \eqref{gamma:leftmost} and \eqref{gamma:rightmost} in \eqref{gamma:bound},
we obtain again that
\[
\| N(\mathbf f,\mathbf x) \| \le \frac{ \gamma \|\mathbf x\|^2}{\psi(\gamma \|\mathbf x\|)}
.
\]
This time, this means that 
\[
\| \mathbf x_{i+1} \| \gamma
\le
\delta + \| N(\mathbf f,\mathbf x) \|\gamma 
\le \delta + \frac{ \gamma^2 \|\mathbf x\|^2}{\psi(\gamma \|\mathbf x\|)}
.
\]

By induction that $\|\mathbf x_i - \zeta\| \gamma(\mathbf f,\zeta) < u_i$
and we are done.
\end{proof}

\begin{exercise}\label{ex:alpha:series}
Consider the following series, defined in 
$\mathbb C^2$:
\[
g(x) = \sum_{i=0}^{\infty} x_1^i x_2^i
.
\]
Compute its radius of convergence. What is its
domain of absolute convergence ?
\end{exercise}

\begin{exercise}The objective of this exercise is to produce
a non-optimal algorithm to approximate $\sqrt{y}$. In order to do that,
consider the mapping $f(x) = x^2 - y$.
\begin{enumerate}
\item Compute $\gamma(f,x)$.
\item Show that for $1 \le y \le 4$, $x_0 = 1/2 + y/2$ is
an approximate zero of the first kind for $x$, associated
to $y$.
\item Write down an algorithm to approximate $\sqrt{y}$ up to
relative accuracy $2^{-63}$.
\end{enumerate}
\end{exercise}

\begin{exercise} \label{ex:separation}
Let $\mathbf f$ be an analytic map between Banach spaces,
and assume that $\zeta$ is a nondegenerate zero of $\mathbf f$.
\begin{enumerate}
\item Write down the Taylor series of 
$D\mathbf f(\zeta)^{-1} \left(\mathbf f(\mathbf x) - \mathbf f(\zeta) \right)$.
\item Show that if $\mathbf f(\mathbf x) = 0$, then
\[
\gamma(\mathbf f, \mathbf \zeta) \|\mathbf x - \zeta \| \ge 1/2 .
\]
\end{enumerate}
This shows that two nondegenerate zeros cannot be at a distance
less than $1/2\gamma(\mathbf f,\zeta)$. Results of this type
appeared in~\ocite{Dedieu-separation}, but some of them were
known before~\ocite{Malajovich-PhD}*{Th.16}.
\end{exercise}

\section{Estimates from data at a point}

Theorem~\ref{th:gamma} guarantees quadratic convergence in a neighborhood
of a known zero $\zeta$. In practical situations, $\zeta$ is
not known. A major result in alpha-theory is the criterion to detect an
approximate zero with just local information. 
We need to slightly modify the definition.

\begin{definition}[Approximate zero of the second kind]
\indexar{approximate zero!of the second kind}
Let $\mathbf f: \mathcal D \subseteq \mathbb E \rightarrow \mathbb F$ be as above.
An {\bf approximate zero of the second
kind} associated to $\mathbf \zeta \in \mathcal D$, $\mathbf f(\zeta)=0$, 
is a point $\mathbf x_0 \in
\mathcal D$, such that
\begin{enumerate}
\item The sequence $(\mathbf x)_{i}$ defined inductively by
$\mathbf x_{i+1} = N(\mathbf f, \mathbf x_i)$ is well-defined
(each $\mathbf x_i$ belongs to the domain of $\mathbf f$ and
$D\mathbf f(\mathbf x_i)$ is invertible and bounded).
\item 
\[
\| \mathbf x_{i+1} - \mathbf x_i \| \le 2^{-2^i+1} \| \mathbf x_1 - \mathbf x_0 \|
.
\]
\item $\lim_{i \rightarrow \infty} \mathbf x_i = \zeta$.
\end{enumerate}
\end{definition}

For detecting approximate zeros of the second kind, we need:

\begin{definition}[Smale's $\beta$ and $\alpha$ invariants]
\indexar{Smale's invariants!beta}
\indexar{Smale's invariants!alpha}
\glossary{$\beta(\mathbf f, \mathbf x)$&--&Invariant related to Newton iteration.}
\glossary{$\alpha(\mathbf f, \mathbf x)$&--&Invariant related to Newton iteration.}
\[
\beta(\mathbf f, \mathbf x) = \| D\mathbf f(\mathbf x)^{-1}
\mathbf f(\mathbf x) \|
\text{ and }
\alpha(\mathbf f, \mathbf x)
=
\beta(\mathbf f, \mathbf x)
\gamma(\mathbf f, \mathbf x)
.
\]
\end{definition}

The $\beta$ invariant can be interpreted as the size of
the Newton step $N(\mathbf f, \mathbf x) - \mathbf x$.

\begin{theorem}[Smale]
\indexar{approximate zero!of the second kind}
\indexar{theorem!Smale}
\indexar{theorem!alpha}
\label{th:alpha}
Let $\mathbf f: \mathcal D \subseteq \mathbb E \rightarrow \mathbb F$ be an analytic
map between Banach spaces. 
Let
\[
\alpha 
\le
\alpha_0 = \frac{13 - 3 \sqrt{17}}{4}.
\]
\glossary{$\alpha_0$&--&The constant $\frac{13 - 3 \sqrt{17}}{4}$.}
Define
\[
r_0 = 
\frac{1+\alpha-\sqrt{1-6\alpha+\alpha^2}}{4\alpha} 
\text{ and }
r_1 =
\frac{1-3\alpha-\sqrt{1-6\alpha+\alpha^2}}{4\alpha} 
.
\]
\glossary{$r_0(\alpha)$&--&The function $\frac{1+\alpha-\sqrt{1-6\alpha+\alpha^2}}{4\alpha}$.}
\glossary{$r_1(\alpha)$&--&The function $\frac{1-3\alpha-\sqrt{1-6\alpha+\alpha^2}}{4\alpha}$.}
Let $\mathbf x_0 \in \mathcal D$ be such that
$\alpha( \mathbf f, \mathbf x_0 ) \le \alpha$ and
assume furthermore that $B(\mathbf x_0, r_0 
\beta(\mathbf f, \mathbf x_0)
) \subseteq \mathcal D$.
Then, 
\begin{enumerate}
\item $\mathbf x_0$ is an approximate zero of the second kind,
associated to some zero $\zeta \in \mathcal D$ of $\mathbf f$.
\item Moreover, $ \| \mathbf x_0 - \zeta \| \le r_0
\beta(\mathbf f, \mathbf x_0)$.
\item Let $\mathbf x_1 = N(\mathbf f, \mathbf x_0)$. Then
$
\| \mathbf x_1 - \zeta \| \le 
r_1
\beta(\mathbf f, \mathbf x_0)$.
\end{enumerate}
The constant $\alpha_0$ is the largest possible with those
properties.
\end{theorem}

This theorem appeared in~\ocite{Smale-PE}. The value for $\alpha_0$ was found by
Wang Xinghua~\ocite{Xinghua}. 
Numerically,
\[
\alpha_0 = 0.157,670,780,786,754,587,633,942,608,019\cdots
\]
Other useful numerical bounds, under the hypotheses of the theorem,
are:
\[
r_0 \le 1.390,388,203\cdots
\text{ and }
r_1 \le 0.390,388,203\cdots
.
\]

\medskip
\par
The proof of Theorem~\ref{th:alpha} follows from the same method as the one
for Theorem~\ref{th:gamma}. We first define the `worst' real function with
respect to Newton iteration. Let us fix $\beta, \gamma > 0$. Define
\[
h_{\beta\gamma}(t) = 
\beta - t + \frac{ \gamma t^2 }{1 - \gamma t}
=
\beta - t + \gamma t^2 + \gamma^2 t^3 + \cdots
.
\]

We assume for the time being that $\alpha = \beta \gamma < 3-2\sqrt{2} = 0.1715\cdots$.
This guarantees that $h_{\beta\gamma}$ has two distinct zeros 
$\zeta_1 = \frac{1+\alpha-\sqrt{\Delta}}{4\gamma}$ and
$\zeta_2 = \frac{1+\alpha+\sqrt{\Delta}}{4\gamma}$
with of course $\Delta = (1+\alpha)^2-8\alpha$.
An useful expression is the product formula
\begin{equation}\label{alpha:product}
h_{\beta\gamma}(x) 
=
2 \frac{ (x-\zeta_1) (x-\zeta_2)}{\gamma^{-1}-x}
.
\end{equation}

From \eqref{alpha:product}, $h_{\beta\gamma}$ has
also a pole at $\gamma^{-1}$. We have always $0< \zeta_1 < \zeta_2 < \gamma^{-1}$.

The function $h_{\beta\gamma}$ is, among the functions with $h'(0)=-1$ and 
$\beta(h,0)\le \beta$ and
$\gamma(h,0)\le \gamma$, the one that has the first zero $\zeta_1$ 
furthest away from the origin. 

\begin{figure}
\centerline{\resizebox{\textwidth}{!}{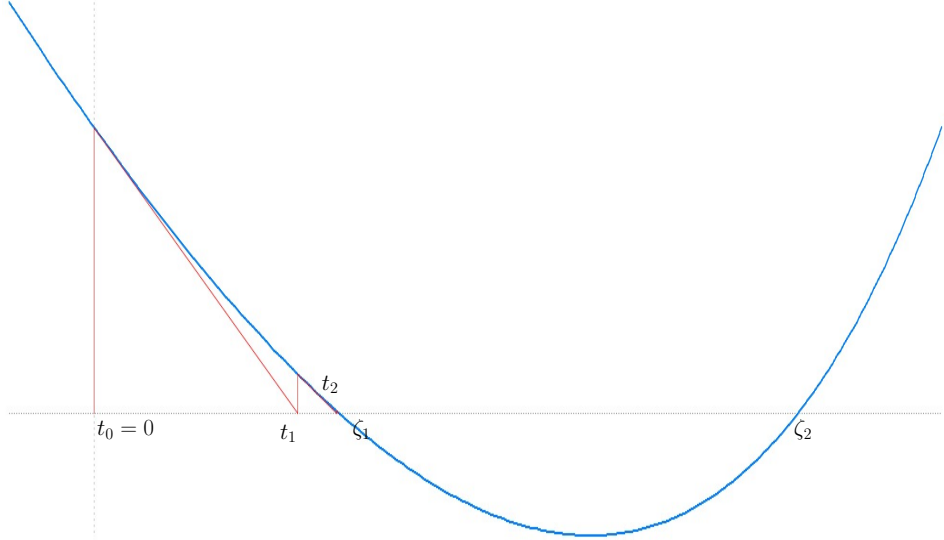}}
\caption{$y=h_{\beta\gamma}(t)$.
\label{graph:alpha:h:beta:gamma}} 
\end{figure}
\medskip

\begin{proposition}\label{prop:alpha:1}
Let $\beta, \gamma >0$, with 
$\alpha = \beta \gamma \le 3 -2 \sqrt{2}$.
let $h_{\beta\gamma}$ be as above. Define
recursively $t_0 = 0$ and $t_{i+1} = N(h_{\beta \gamma}, t_i)$.
then 
\begin{equation}
\label{alpha:x}
t_i = \zeta_1 
\frac{1-q^{2^i-1}}{1-\eta q^{2^i-1}},
\end{equation}
with
\[
\eta = \frac{\zeta_1}{\zeta_2} = \frac{1+\alpha-\sqrt{\Delta}}{1+\alpha+\sqrt{\Delta}}
\text{ and }
q= \frac{ \zeta_1 - \gamma \zeta_1 \zeta_2}{\zeta_2 - \gamma \zeta_1 \zeta_2}
= 
\frac{1-\alpha-\sqrt{\Delta}}{1-\alpha+\sqrt{\Delta}}
.
\]
\end{proposition}

\begin{proof}

By differentiating \eqref{alpha:product}, one obtains
\[
h'_{\beta\gamma}(t) 
=
h_{\beta\gamma}(t) 
\left(
\frac{1}{t-\zeta_1}
+
\frac{1}{t-\zeta_2}
+
\frac{1}{\gamma^{-1}-t}
\right)
\]
and hence the Newton operator is
\[
N(h_{\beta\gamma}, t) =
t - \frac{1}{ 
\frac{1}{t-\zeta_1}
+
\frac{1}{t-\zeta_2}
+
\frac{1}{\gamma^{-1}-t}
}
.
\]

A tedious calculation shows that $N(h_{\beta \gamma}, t)$ is a rational 
function of degree 2. Hence, it is defined by 5 coefficients,
or by 5 values. 

In order to solve the recurrence for $t_i$, 
we change coordinates using a fractional linear
transformation. As the Newton operator will have two attracting
fixed points ($\zeta_1$ and $\zeta_2$), we will map those points to
$0$ and $\infty$ respectively. For convenience, we will map
$t_0=0$ into $y_0=1$. Therefore, we set
\[
S(t)=\frac{ \zeta_2 t - \zeta_1 \zeta_2 }
{\zeta_1 t - \zeta_1 \zeta_2}
\text{\hspace{1em} and \hspace{1em}}
S^{-1}(y)=\frac{ - \zeta_1 \zeta_2 y + \zeta_1 \zeta_2}
{-\zeta_1 y + \zeta_2}
\]
Let us look at the sequence $y_i = S(t_i)$. By construction
$y_0 = 1$, and subsequent values are given by the recurrence
\[
y_{i+1} = S( N(h_{\beta\gamma}, S^{-1}(y_i))).
\]

It is an exercise to check that 
\begin{equation}
y_{i+1} = q y_i^2,
\end{equation}

Therefore we have $y_i = q^{2^i-1}$, and equation \eqref{alpha:x} holds.
\end{proof}
\medskip
\par

\begin{proposition}\label{prop:alpha:2}
Under the conditions of Proposition~\ref{prop:alpha:1},
0 is an approximate zero of the second kind for 
$h_{\beta\gamma}$ if and only if
\[
\alpha = \beta \gamma \le \frac{13 - 3 \sqrt{17}}{4}.
\]
\end{proposition}

\begin{proof}
Using the closed form for $t_i$, we get:
\begin{eqnarray*}
t_{i+1}-t_i
&=&
\frac{ 1 - q^{2^{i+1}-1}}{1-\eta q^{2^{i+1}-1}}
-\frac{ 1 - q^{2^{i}-1}}{1-\eta q^{2^{i}-1}}
\\
&=&
q^{2^{i}-1}
\frac{
(1-\eta)(1-q^{2^i})
}
{
(1-\eta q^{2^{i+1}-1})
(1-\eta q^{2^{i}-1})
}
\end{eqnarray*}

In the particular case $i=0$,
\[
t_1 - t_0 =
\frac{
1-q
}
{
1-\eta q
}
=\beta
\]

Hence 
\[
\frac{t_{i+1}-t_i}{\beta} =
C_i q^{2^i-1}
\]
with 
\[
C_i =
\frac{
(1-\eta)(1-\eta q)(1-q^{2^i})
}
{
(1-q)
(1-\eta q^{2^{i+1}-1})
(1-\eta q^{2^{i}-1})
}
.
\]

Thus, $C_0 = 1$. The reader shall verify in
Exercise~\ref{ex:alpha:C} that $C_i$ is 
a non-increasing sequence. Its limit is non-zero.

From the above, it is clear that $0$ is an approximate zero of the second kind
if and only if
$q \le 1/2$. Now, if we clear denominators and rearrange terms in
$(1+\alpha-\sqrt{\Delta})/(1+\alpha+\sqrt{\Delta})=1/2$, we
obtain the second degree polynomial
\[
2 \alpha^2 - 13 \alpha + 2 = 0
.
\]

This has solutions $(13 \pm \sqrt{17})/2$. When $0 \le \alpha \le \alpha_0=
(13-\sqrt{17})/2$, the polynomial values are positive and hence 
$q \le 1/2$.
\end{proof}

\begin{proof}[Proof of Th.\ref{th:alpha}]
Let $\beta=\beta(\mathbf f, \mathbf x_0)$ and $\gamma = \gamma(\mathbf f,
\mathbf x_0)$. Let $h_{\beta\gamma}$ and the sequence $t_i$ be
as in Proposition~\ref{prop:alpha:1}. 
By construction, $\|\mathbf x_1 - \mathbf x_0\| = \beta = t_1-t_0$.
We use the following notations:
\[
\beta_i = \beta(\mathbf f, \mathbf x_i)
\text{ and  }
\gamma_i = \gamma(\mathbf f, \mathbf x_i)
.
\]
Those will be compared to
\[
\hat \beta_i = \beta(h_{\beta\gamma},t_i))
\text{ and  }
\hat \gamma_i = \gamma(h_{\beta\gamma},t_i))
.
\]
\noindent
{\bf Induction hypothesis:} 
$
\beta_i \le \hat \beta_i
$ and for all $l \ge 2$, 
\[
\| D\mathbf f(\mathbf x_i)^{-1} D^l\mathbf f(\mathbf x_{i}) \| \le
- \frac{ h_{\beta\gamma}^{(l)}(t_i)}{h_{\beta\gamma}'(t_i)}
.
\]
\medskip

The initial case when $i=0$ holds by construction.
So let us assume that the hypothesis holds for $i$.
We will estimate
\begin{equation}\label{eq:alpha:beta}
\beta_{i+1}  
\le
\| D\mathbf f(\mathbf x_{i+1})^{-1} D\mathbf f(\mathbf x_{i}) \|
\| D\mathbf f(\mathbf x_{i})^{-1} \mathbf f(\mathbf x_{i+1}) \|
\end{equation}
and
\begin{equation}\label{eq:alpha:gamma}
\gamma_{i+1}  
\le
\| D\mathbf f(\mathbf x_{i+1})^{-1} D\mathbf f(\mathbf x_{i}) \|
\frac{\| D\mathbf f(\mathbf x_{i})^{-1} D^k\mathbf f(\mathbf x_{i+1}) \|}{k!}
.
\end{equation}

By construction, $\mathbf f(\mathbf x_i) + D\mathbf f(\mathbf x_i) (\mathbf x_{i+1}-\mathbf x_i) = 0$.
The Taylor expansion of $\mathbf f$ at $\mathbf x_i$ is therefore
\[
D\mathbf f(\mathbf x_{i})^{-1} \mathbf f(\mathbf x_{i+1}) = 
\sum_{k \ge 2} 
\frac{D\mathbf f(\mathbf x_{i})^{-1} D^k\mathbf f(\mathbf x_{i}) (\mathbf x_{i+1}-\mathbf x_i)^k}{k!}
\]

Passing to norms,
\[
\|D\mathbf f(\mathbf x_{i})^{-1} \mathbf f(\mathbf x_{i+1})\| \le \frac{\beta_i^2\gamma_i}{1 - \gamma_i}
\]

The same argument shows that
\[
-
\frac{h_{\beta\gamma}(t_{i+1})}{h_{\beta\gamma}'(t_i)} 
=
\frac{\beta(h_{\beta\gamma},t_i)^2\gamma(h_{\beta\gamma},t_i)}{1 - \gamma(h_{\beta\gamma},t_i)}
\]

From Lemma~\ref{lem:gamma:derivative},
\[
\| D\mathbf f(\mathbf x_{i+1})^{-1} D\mathbf f(\mathbf x_{i})\| \le 
\frac{(1- \beta_i \gamma_i)^2}{\psi(\beta_i \gamma_i)}
.
\]

Also, computing directly,
\begin{equation}\label{alpha:derivative}
\frac{h_{\beta\gamma}'(t_{i+1})}{h_{\beta\gamma}'(t_i)} 
=
\frac{(1- \hat \beta \hat \gamma)^2}
{\psi(\hat \beta \hat \gamma)}
.
\end{equation}

We established that
\[
\beta_{i+1}
\le
\frac{\beta_i^2 \gamma_i (1- \beta_i \gamma_i)}{\psi(\beta_i \gamma_i)}
\le
\frac{\hat \beta_i^2 \hat \gamma_i (1- \hat \beta_i \hat \gamma_i)}{\psi(\hat \beta_i \hat \gamma_i)}
=
\hat
\beta_{i+1}
.
\]

Now the second part of the induction hypothesis:

\[
D\mathbf f(\mathbf x_i)^{-1} D^l\mathbf f(\mathbf x_{i+1}) = 
\sum_{k \ge 0} 
\frac{1}{k!}
\frac{D\mathbf f(\mathbf x_i)^{-1} D^{k+l} \mathbf f(\mathbf x_i) (\mathbf x_{i+1}-\mathbf x_i)^k}
{k+l}
\]

Passing to norms and invoking the induction hypothesis,
\[
\|
D\mathbf f(\mathbf x_i)^{-1} D^l\mathbf f(\mathbf x_{i+1}) 
\|
\le 
\sum_{k \ge 0} 
- \frac{ h_{\beta\gamma}^{(k+l) 
}(t_i)
\hat \beta_i^k 
}{k! h_{\beta\gamma}'(t_i)}
\]
and then using Lemma~\ref{lem:gamma:derivative} and \eqref{alpha:derivative},
\[
\|
D\mathbf f(\mathbf x_{i+1})^{-1} D^l\mathbf f(\mathbf x_{i+1}) 
\|
\le 
\frac{(1-\hat \beta_i \hat \gamma_i)^2}{\psi(\hat \beta_i \hat \gamma_i)}
\sum_{k \ge 0} 
- \frac{ h_{\beta\gamma}^{(k+l) 
}(t_i)
\hat \beta_i^k 
}{k! h_{\beta\gamma}'(t_i)}
.
\]

A direct computation similar to \eqref{alpha:derivative} shows that
\[
- \frac{ h_{\beta\gamma}^{(k+l) }(t_{i+1})}{k! h_{\beta\gamma}'(t_{i+1})}
= 
\frac{(1-\hat \beta_i \hat \gamma_i)^2}{\psi(\hat \beta_i \hat \gamma_i)}
\sum_{k \ge 0} 
- \frac{ h_{\beta\gamma}^{(k+l) 
}(t_i)
\hat \beta_i^k 
}{k! h_{\beta\gamma}'(t_i)}
.
\]
and since the right-hand-terms of the last two equations are equal,
the second part of the induction hypothesis proceeds.
Dividing by $l!$, taking $l-1$-th roots and maximizing over
all $l$, we deduce that $\gamma_{i} \le \hat \gamma_i$.
\medskip

Proposition~\ref{prop:alpha:2} then implies that 
$\mathbf x_0$ is an approximate zero.
\medskip

The second and third statement follow respectively from
\[
\| \mathbf x_0 - \zeta \| \le \beta_0 + \beta_1 + \cdots = \zeta_1
\]
and
\[
\| \mathbf x_1 - \zeta \| \le \beta_1 + \beta_2 + \cdots = \zeta_1 - \beta.
\]
\end{proof}

The same issues as in Theorem~\ref{th:gamma} arise.
First of all, we actually proved a sharper statement. 
Namely,

\begin{theorem}
\indexar{theorem!alpha!sharp}
\label{th:alpha:sharp}
Let $\mathbf f: \mathcal D \subseteq \mathbb E \rightarrow \mathbb F$ be an analytic
map between Banach spaces. 
Let
\[
\alpha 
\le
3-2 \sqrt{2}.
\]
Define
\[
r = 
\frac{1+\alpha-\sqrt{1-6\alpha+\alpha^2}}{4\alpha} 
.
\]
Let $\mathbf x_0 \in \mathcal D$ be such that
$\alpha( \mathbf f, \mathbf x_0 ) \le \alpha$ and
assume furthermore that $B(\mathbf x_0, r 
\beta(\mathbf f, \mathbf x_0)
) \subseteq \mathcal D$.
Then, the sequence $\mathbf x_{i+1} = N(\mathbf f, \mathbf x_i)$
is well defined, and there is a zero $\zeta \in \mathcal D$ of
$\mathbf f$ such that
\[
\|\mathbf x_i - \zeta \|
\le 
q^{2^i-1} \frac{1-\eta}{1-\eta q^{2^i-1}} 
r \beta (\mathbf f, \mathbf x_0)
.
\]
for $\eta$ and $q$ as in Proposition~\ref{prop:alpha:1}.
\end{theorem}

Table~\ref{table:alpha:alpha-conv} and Figure~\ref{graph:alpha:alpha-conv} show how
fast $\|\mathbf x_i - \zeta\|/\beta$ decreases in terms of $\alpha$ and $i$.

\begin{table}
\centerline{
\begin{tabular}{|c|r@{.}lr@{.}lr@{.}lr@{.}lr@{.}l|}
\hline&\multicolumn{2}{c}{$1/32$}&\multicolumn{2}{c}{$1/16$}&\multicolumn{2}{c}{$1/10$}&\multicolumn{2}{c}{$1/8$}&\multicolumn{2}{c|}{$\frac{13-3\sqrt{17}}{4}$}\\\hline
$ 1$&  4&854&  3&683&  2&744&  2&189&  1&357\\
$ 2$& 14&472& 10&865&  7&945&  6&227&  3&767\\
$ 3$& 33&700& 25&195& 18&220& 14& 41&  7&874\\
$ 4$& 72&157& 53&854& 38&767& 29&648& 15&881\\
$ 5$&149& 71&111&173& 79&861& 60&864& 31&881\\
$ 6$&302&899&225&811&162& 49&123&295& 63&881\\
\hline\end{tabular}
}
\caption{Values of $-log_2 (\|\mathbf x_i - \zeta\|/\beta)$ 
in function of $\alpha$ and $i$. \label{table:alpha:alpha-conv}} 
\end{table}

\begin{figure}
\centerline{\resizebox{\textwidth}{!}{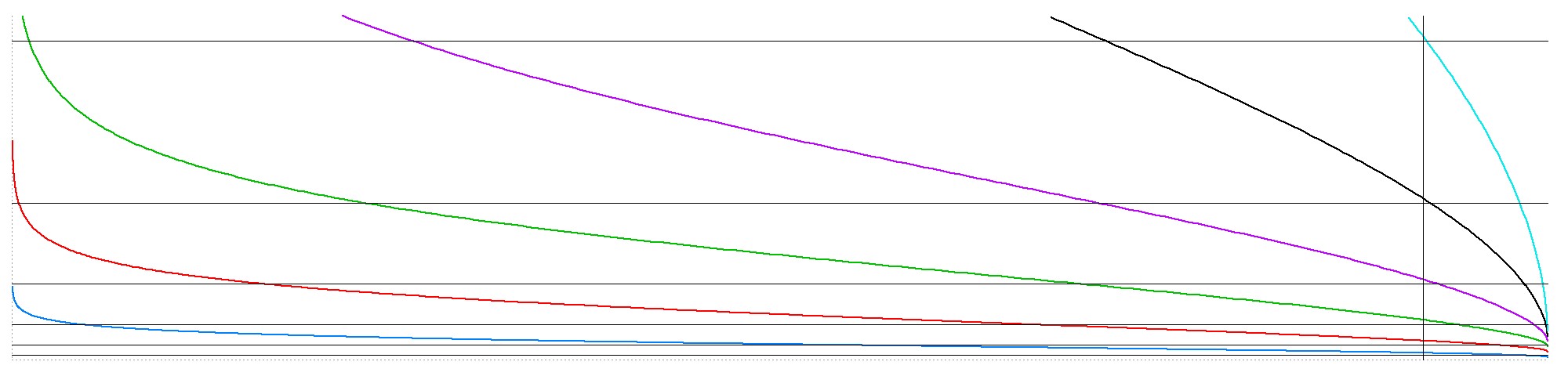}}
\caption{Values of $-log_2 (\|\mathbf x_i - \zeta\|/\beta)$ in function of 
$\alpha$ for
$i=1$ to $6$. \label{graph:alpha:alpha-conv}} 
\end{figure}
\medskip

The final issue is robustness. There is no obvious modification
of the proof of Theorem~\ref{th:alpha} to provide a nice statement,
so we will rely on Theorem~\ref{th:gamma:robust} indeed.
\medskip

\begin{theorem}\label{th:alpha:robust}
\indexar{theorem!alpha!robust}
Let $\mathbf f: \mathcal D \subseteq 
\mathbb E \rightarrow \mathbb F$ be an analytic
map between Banach spaces. 
Let $\delta$, $\alpha$ and $u0$ satisfy
\[
0 \le 2 \delta < u_0 = \frac{r \alpha}{(1-r\alpha)\psi(r\alpha)}
<
2 - \frac{\sqrt{14}}{2}
\]
with $r = 
\frac{1+\alpha-\sqrt{1-6\alpha+\alpha^2}}{4\alpha}$.
Assume that
\begin{enumerate}
\item
\[
B=B\left(\mathbf x_0, 2r \beta(\mathbf f, \mathbf x_0)\right) 
\subseteq \mathcal D.
\]
\item $\mathbf x_0 \in B$, and the sequence $\mathbf x_i$ satisfies
\[
\| \mathbf x_{i+1} - N(\mathbf f, \mathbf x_i) \| \frac{r \beta(f,x_0)}
{(1 - r\alpha)\psi(r\alpha)} 
\le
\delta
\]
\item The sequence $u_i$ is defined inductively by 
\[
u_{i+1} = \frac { u_i^2 }{\psi (u_i)} + \delta
.
\]
\end{enumerate}

Then the sequences $u_i$ and $\mathbf x_i$ are well-defined for all $i$, 
$\mathbf x_i \in \mathcal D$, and 
\[
\frac { \left\| \mathbf x_i - \zeta \right\| }{ \left\| \mathbf x_1 - \mathbf x_0 \right\| }
\le
\frac{r u_i}{u_0}
\le
r \max\left( 2^{-2^i+1}, 2 \frac{\delta}{u_0} \right).
\]
\end{theorem}

Numerically, $\alpha_0 = 0.074,290\cdots$ satisfies the hypothesis
of the Theorem. A version of this theorem (not as sharp, and another
metric) appeared as Theorem~2 in~\ocite{Malajovich94}. 

The following Lemma will be useful:

\begin{lemma}\label{lem:gamma:highder}
Assume that $u=\gamma(\mathbf f, \mathbf x) \| \mathbf x - \mathbf y\|
\le 1 - \sqrt{2}/2$. Then,
\[
\gamma(\mathbf f, \mathbf y) \le 
\frac{ \gamma(\mathbf f, \mathbf x) }{(1-u)\psi(u)} 
.
\]
\end{lemma}

\begin{proof}
In order to estimate the higher derivatives, we expand:

\[
\frac{1}{l!}
D\mathbf f(\mathbf x)^{-1} D^l\mathbf f(\mathbf y) = 
\sum_{k \ge 0} 
\binomial{k+l}{l}
\frac{D\mathbf f(\mathbf x)^{-1} D^{k+l} \mathbf f(\mathbf x) (\mathbf y-\mathbf x)^k}
{k+l}
\]
and by Lemma~\ref{one:over:one:minus:t} for $d=l+1$,
\[
\frac{1}{l!}
\| D\mathbf f(\mathbf x)^{-1} D^l\mathbf f(\mathbf y)\|  \le
\frac{ \gamma(\mathbf f, \mathbf x)^{l-1} }{(1-u)^{l+1}}
.
\]

Combining with Lemma~\ref{lem:gamma:derivative},
\[
\frac{1}{l!}
\| D\mathbf f(\mathbf y)^{-1} D^l\mathbf f(\mathbf y)\|  \le
\frac{ \gamma(\mathbf f, \mathbf x)^{l-1} }{(1-u)^{l-1} \psi(u)}
.
\]

Taking the $l-1$-th power,
\[
\gamma(\mathbf f, \mathbf y) \le
\frac{ \gamma(\mathbf f, \mathbf x) }{(1-u) \psi(u)}
.
\]
\end{proof}

\begin{proof}[Proof of Theorem~\ref{th:alpha:robust}]
We have necessarily $\alpha < 3-2\sqrt{2}$ or $r$ is 
undefined. Then (Theorem~\ref{th:alpha:sharp}) there
is a zero $\zeta$ of $\mathbf f$ with 
$\|\mathbf x_0 - \zeta\| \le r \beta(f,x_0)$. 
Then, Lemma~\ref{lem:gamma:highder}
implies that $\|\mathbf x_0-\zeta\| \gamma(\mathbf f, \mathbf \zeta) \le u_0$.
Now apply Theorem~\ref{th:gamma:robust}. 

\end{proof}

\begin{exercise}\label{ex:alpha:C}The objective of this
exercise is to show that $C_i$ is non-increasing.
\begin{enumerate}
\item Show the following trivial lemma: {\bf If $0 \le s < a \le b$,
then $\frac{a-s}{b-s} \le \frac{a}{b}$}.
\item Deduce that $q \le \eta$.
\item Prove that $C_{i+1}/C_i \le 1$.
\end{enumerate}
\end{exercise}

\begin{exercise} \label{ex:gamma:zeta}
Show that 
\[
\zeta_1 \gamma(\zeta_1) = 
\frac{1+\alpha-\sqrt{\Delta}}{3-\alpha+\sqrt{\Delta}}
\frac{1}{\psi \left( \frac{1+\alpha-\sqrt{\Delta}}{4} \right)}.
\]
\end{exercise}

\part{Inclusion and exclusion}
\section{Eckart-Young theorem}
\label{sec:EY}

The following classical theorem in linear algebra is known as the
{\bf singular value decomposition} ({\bf svd} for short).

\begin{theorem}\label{th:svd} 
\indexar{singular value decomposition}
Let $A : \mathbb R^n \mapsto \mathbb R^m$ (resp. 
$\mathbb C^n \rightarrow \mathbb C^m$) be linear. Then,
there are $\sigma_1 \ge \dots \ge \sigma_r > 0$, $r \le m,n$, such that
\[
A = U \Sigma V^*
\]
with $U \in O(m)$ (resp. $U (m)$), $V \in O(n)$ (resp. $U (n)$) and 
$\Sigma_{ij}=\sigma_i$ for $i=j\le r$ and $0$ otherwise.
\end{theorem}
\glossary{$\sigma_1, \dots, \sigma_n$&--&Singular values associated to a matrix.}

It is due to Sylvester (real $n\times n$ matrices) and to
Eckart and Young \ycite{Eckart-Young} in the general case,
now exercise~\ref{ex:svd} below.

$\Sigma$ is a $m \times n$ matrix. It is possible to rewrite this  
in an `economical’ formulation with $\Sigma$ an $r \times r$ matrix,
$U$ and $V$ orthogonal (resp. unitary) $m\times r$ and $n \times r$
matrices.
The numbers $\sigma_1, \dots, \sigma_r$ are called 
{\bf singular values} of A. They may
be computed by extracting the positive square root of the non-zero
eigenvalues of $A^* A$ or $AA^*$, whatever matrix is smaller. 
The operator and
Frobenius norm of $A$ may be written in terms of the $\sigma_i$'s:
\[
\|A\|_2 = \sigma_1 \hspace{3em} \|A\|_F = \sqrt{ \sigma_1^2 + \cdots + \sigma_r^2}
.
\]

The discussion and the results above hold when $A$ is a linear operator
between finite dimensional
inner product spaces. It suffices to choose an orthonormal
basis, and apply Theorem~\ref{th:svd} to the corresponding matrix.

\medskip
\par

When $m = n = r$, $\|A^{-1}\|_2 = \sigma_n$. In this case, the {\bf condition
number} of A for linear solving is defined as

\[
\kappa(A) = 
\|A\|_* \| A^{-1}\|_{**}.
\]
The choice of norms is arbitrary, as long as operator and vector norms
are consistent.
Two canonical choices are
\[
\kappa_2 (A) = \|A\|_2 \|A^{-1}\|_2
\text{\ and\ }
\kappa_D (A) = \|A\|_F \|A^{-1}\|_2.
\]

\indexar{condition number!for linear equations}
The second choice
was suggested by Demmel~\ycite{DEMMEL-CONDITION}. Using that
definition he obtained bounds on the probability that
a matrix is poorly conditioned. The exact probability distribution 
for the most usual probability measures in matrix space was
computed in~\ocite{EDELMAN}.
\medskip

Assume that $A(t)\mathbf x(t) \equiv \mathbf b(t)$ is a 
family of
problems and solutions depending smoothly on a parameter $t$.
Differentiating implicitly,

\[
\dot A\mathbf x + A \dot{\mathbf x} = \dot{\mathbf b}
\]

which amounts to
\[
\dot {\mathbf x} = A^{-1}\dot{\mathbf b} - A^{-1} \dot A \mathbf x.
\]

Passing to norms and to relative errors, we quickly obtain
\[
\frac{\|\dot {\mathbf x}\|}
{\| \dot{\mathbf x}\|}
 \le \kappa_D(A) 
\left( 
\frac{\|\dot A\|_F}{\|A\|_F}
+
\frac{\|\dot {\mathbf b}\|}{\|\mathbf b\|}
\right)
.\]

This bounds the relative error in the solution
$\mathbf x$ in terms of the relative error in the coefficients. The usual
paradigm in numerical linear algebra 
dates from ~\ocite{Turing} and \ocite{Wilkinson}.
After the
rounding-off during computation, we obtain the
exact solution of a perturbed system.
Bounds for the perturbation or {\bf backward error}
are found through line by line analysis of
the algorithm. The output error or {\bf forward error}
is bounded by the backward error,
times the condition number.

Condition numbers provide therefore an important metric 
invariant for numerical analysis problems. A geometric 
interpretation in the case of linear equation solving is:

\begin{theorem}
\indexar{theorem!condition number!linear}
Let $A$ be a nondegenerate square matrix.
\[
\|A^{-1}\|_2 = \min_{\det(A+B)=0} \|B\|_F
\]
\end{theorem}

In particular, this implies that
\[
\kappa_D (A)^{-1} = \min_{\det(A+B)=0} \frac{\|B\|_F}{\|A\|_F}
\]

A pervading principle in the subject is: {\bf the
inverse of the condition number is related to the distance to the 
ill-posed problems}.
\medskip
\par

It is possible to define the condition number for a full-rank non-
square matrix by
\[
\kappa_D (A) =
\|A\|_F
\ 
\sigma_{\min(m,n)}(A)^{-1}
.
\]

\begin{theorem}\label{th:eckart-young}
\cite{Eckart-Young2}
\indexar{theorem!Eckart-Young}
Let $A$ be an $m \times n$ matrix of rank $r$.
Then,
\[
\sigma_r(A)^{-1} =
\min_{\sigma_r(A+B)=0} {\|B\|_F}
.
\]
In particular, if $r=\min(m,n)$,
\[
\kappa_D (A)^{-1} = \min_{\sigma_r(A+B)=0} \frac{\|B\|_F}{\|A\|_F}
. 
\]
\end{theorem}

\begin{exercise} \label{ex:svd} Prove Theorem~\ref{th:svd}.
Hint: let $u$, $v$, $\sigma$ such that
$A v =\sigma u$ with $\sigma$ maximal, $\|u\|=1$,
$\|v\|=1$. What can you say about $A_{|v^\perp}$? 
\end{exercise}
\begin{exercise} Prove Theorem~\ref{th:eckart-young}.
\end{exercise}

\begin{exercise}
Assume furthermore that $m < n$. Show that the same
interpretation for the condition number still holds, namely the norm
of the perturbation of {\bf some} solution is bounded by the condition number,
times the perturbation of the input.
\end{exercise}

\section{The space of homogeneous polynomial systems}

We will denote by $\mathcal H_d^{\mathbb R}$ the space of polynomials
of degree $d$ in $n+1$ variables. This space can be assimilated 
to the space of symmetric $d$-linear forms. For instance, 
when $d=2$, the polynomial 
\[
f(x_0, x_1) = f_0 x_0^2 + f_1 x_0 x_1 + f_2 x_1^2
= 
\left[
\begin{matrix}
x_0 & x_1
\end{matrix}
\right]
\left[
\begin{matrix}
f_0 & f_1/2 \\
f_1/2 & f_0
\end{matrix}
\right]
\left[
\begin{matrix}
x_0 \\ x_1
\end{matrix}
\right]
\]
can be assimilated to a symmetric bilinear form and can be represented
by a matrix. In general, a homogeneous polynomial can be represented
by a symmetric tensor
\[
f(\mathbf x) = \sum_{|\mathbf a|=d} f_{\mathbf a} x_0^{a_0} \cdots x_n^{a_n}
=
\sum_{0 \le i_1, \dots, i_d \le n}
T_{i_1 i_2 \dots i_d} x_{i_1} x_{i_2} \cdots x_{i_d}
\]
where
\[
f_{\mathbf a} = \sum_{\mathbf a=\mathrm e_{i_1}+\mathrm e_{i_2}+ \cdots \mathrm e_{i_d} } T_{i_1 i_2 \dots i_d} 
.
\]

The canonical inner product for tensors is given by
\[
\langle S, T \rangle =
\sum_{0 \le i_1, \dots, i_d \le n}
S_{i_1 i_2 \dots i_d} 
T_{i_1 i_2 \dots i_d} 
\]

The same inner product for polynomials is written
\[
\langle f, g \rangle =
\sum_{|\mathbf a|=d}
\frac { f_{\mathbf a} g_{\mathbf a} }{\binomial{d}{\mathbf a}} .
\]
where $\binomial{d}{\mathbf a}=\frac{d!}{a_0!a_1!\cdots a_n!}$ 
is the coefficient of $(x_0 + \cdots + x_n)^d$
in $x^a$. 

\begin{lemma}\label{lem:invariant} Let $Q$ be an orthogonal $n \times n$ matrix,
that is $Q^T Q = I$. Then,
\[
\langle f \circ Q, g \circ Q \rangle = \langle f, g \rangle
\]
\end{lemma}

\begin{exercise}
Prove Lemma~\ref{lem:invariant}
\end{exercise}

We say that the above inner product is {\bf invariant under 
orthogonal action}. We will always assume this inner-product
for $\mathcal H_d^{\mathbb R}$.
\medskip
\par
It is also important to notice that $\mathcal H_d^{\mathbb R}$ is that
it is a {\bf reproducing kernel space}. Let
\[
K_d(\mathbf x,\mathbf y) = \langle \mathbf x, \mathbf y \rangle^d .
\]

Then 
\[f(\mathbf y) = \langle f(\cdot), K_d(\cdot,\mathbf y)\rangle,
\]
\[
Df(\mathbf y) \mathbf u =
\langle f(\cdot), D_{\mathbf y} K_d(\cdot,\mathbf y) \mathbf u\rangle,
\] 
etc...

\section{The condition number}

Now, let's denote by $\HdR$ the space of systems of homogeneous
polynomials of degree $\mathbf d = (d_1, \dots, d_n)$. The {\bf condition number}
measures how does the solution of an equation depends upon the
coefficients. 

Therefore, assume that both a polynomial system $\mathbf f \in S(\HdR)$
and a point $\mathbf x \in S(\mathbb R^{n+1})$ depend
upon a parameter $t$. Say,
\[
\mathbf f_t(\mathbf x_t) \equiv 0. 
\]

Differentiating, one gets
\[
D\mathbf f_t(\mathbf x_t) \mathbf {\dot x}_t = - \mathbf {\dot f}_t(\mathbf x_t)
\]
so

\begin{equation}\label{sensitivity}
\| \mathbf{\dot x}_t \| \le \| D\mathbf f_t(\mathbf x_t)_{|\mathbf x_t^{\perp}} ^{-1}\| \|\mathbf {\dot f}_t(\mathbf x_t)\| .
\end{equation}

The {\bf normalized condition number} is defined for
$\mathbf f \in \HdR$ and $\mathbf x \in \mathbb R^{n+1}$ as
\[
\mu(\mathbf f,\mathbf x) = \|\mathbf f\| 
\left\| 
\left(
\left[
\begin{matrix}
d_1^{-1/2}\|\mathbf x\|^{-d_1+1}\\
&\ddots\\
&&d_n^{-1/2}\|\mathbf x\|^{-d_n+1}\\
\end{matrix}
\right]
D\mathbf f(\mathbf x)_{|\mathbf x^{\perp}}
\right)^{-1}
\right\|
.
\]

In the special case $\mathbf f \in S(\HdR)$ and $\mathbf x \in S(\mathbb R^{n+1})$,
\[
\mu(\mathbf f,\mathbf x) =  
\left\| 
\left(
\left[
\begin{matrix}
d_1^{-1/2}\\
&\ddots \\
&&d_n^{-1/2}\\
\end{matrix}
\right]
D\mathbf f(\mathbf x)_{|\mathbf x^{\perp}}
\right)^{-1}
\right\|
.
\]

\begin{proposition}\label{prop:mu}\ \\
\begin{enumerate}
\item If $\mathbf f_t$ and $\mathbf x_t$ are paths in 
$S(\HdR)$ and $S(\mathbb R^{n+1})$ respectively,
and 
$\mathbf f_t(\mathbf x_t) \equiv 0$ then
\[
\|\mathbf {\dot x}_t\| \le \mu(\mathbf f_t, \mathbf x_t) \| \mathbf {\dot f_t} \|
.
\]
\item Let $\mathbf x \in S(\mathbb R^{n+1})$ be fixed.
Then the mapping
\[
\funcao{\pi}{\HdR}{L(\mathbf x^{\perp}, \mathbb R^n)}
{\mathbf f}{
\left[
\begin{matrix}
d_1^{-1/2}\\
&d_2^{-1/2}\\
&& \ddots\\
&&&d_n^{-1/2}\\
\end{matrix}
\right]
D\mathbf f(\mathbf x)_{|\mathbf x^{\perp}}
}
\]
restricts to an isometry $\pi_{|(\ker \pi)^{\perp}}: (\ker \pi)^{\perp}
\rightarrow L(\mathbf x^{\perp}, \mathbb R^n)$.
\item
Let $\mathbf f \in S(\HdR)$ and $\mathbf x \in S(\mathbb R^{n+1})$.
Then,
\[
\mu(\mathbf f,\mathbf x) = \frac{1}{\min\{\|\mathbf f-\mathbf g\|: D\mathbf g(\mathbf x)_{|\mathbf x^{\perp}} \text{ singular}\}}.
\]
\item
If furthermore $\mathbf f(\mathbf x)=0$,
\[
\mu(\mathbf f,\mathbf x) = \frac{1}{\min\{\|\mathbf f-\mathbf g\|: \mathbf g(\mathbf x)=0 \text{ and } D\mathbf g(\mathbf x)_{|\mathbf x^{\perp}} \text{ singular}\}}.
\]
\end{enumerate}
\end{proposition}

\begin{proof}
Item 1 follows from \eqref{sensitivity}.
In order to prove item 2,
let $\mathbf x \in S(\mathbb R^{n+1})$ be fixed and let $\mathbf f \in \HdR$.
Assume that
$\mathbf y \perp \mathbf x$. We can write $\mathbf f(\mathbf x+\mathbf y)$ as
\[
\mathbf f(\mathbf x+\mathbf y) = \mathbf f(\mathbf x) + D\mathbf f(\mathbf x)_{|\mathbf x ^{\perp}} \mathbf y + \frac{1}{2} D^2\mathbf f(\mathbf x)_{|\mathbf x^{\perp}}
(\mathbf y-\mathbf x,\mathbf y-\mathbf x) + \cdots
\]

This suggests a decomposition of $\HdR$ into terms
that are `constant', `linear' or `higher order' at $x$. 

\[
\HdR
=
H_0 \oplus H_1 \oplus H_2 \oplus \cdots
.
\]

An orthonormal basis for $H_1$ would be
\[
\left( \frac{1}{\sqrt{d}} \frac{\partial K_{d_i}(\cdot, \mathbf x)}{\partial \mathbf u_j} 
\mathrm{e}_i\right)
\]
where $(\mathbf u_1, \dots, \mathbf u_n)$ is an orthonormal basis of $\mathbf x^{\perp}$ and
$(\mathrm e_1,\dots , \mathrm e_n)$ is the canonical basis of $\mathbb R^n$.

In this basis, the projection of $\mathbf f$ in $H_1$ is
just

\[
\left[
\begin{matrix}
& \vdots & \\
\cdots & 
\left\langle \mathbf f_i,  \frac{1}{\sqrt{d}} \frac{\partial K_{d_i}(\cdot, \mathbf x)}{\partial\mathbf  u_j} \right \rangle
& \cdots\\
& \vdots & 
\end{matrix}
\right]
=
\left[
\begin{matrix}
d_1^{-1/2}\\
&\cdots\\
&&d_n^{-1/2}\\
\end{matrix}
\right]
D\mathbf f(\mathbf x)_{|\mathbf x^{\perp}}
.
\]

Thus, the subspace $H_1$ of $\HdR$ is isomorphic to the space
of $n \times n$ matrices. Moreover, $\pi: \HdR \rightarrow H_1$
is an orthogonal projection. Items 3 and 4 follow now
easily from Theorem~\ref{th:eckart-young}.
\end{proof}

\begin{exercise}\label{ex:mu}
Deduce that for all $\mathbf f \in \HdR$, $0 \ne \mathbf x \in \mathbb R^{n+1}$,
$\mu(\mathbf f, \mathbf x) \ge \sqrt{n}$.
\end{exercise}

We denote by $\rho(\mathbf x,\mathbf y) = \widehat{(\mathbf x0\mathbf y)}$ the angular 
distance between $\mathbf x\in S^n$ and $\mathbf y \in S^n$.
The following estimate is quite useful:

\begin{theorem}\label{th:varmu}
Let $\mathbf f,\mathbf g \in S(\HdR)$ and let $\mathbf x,\mathbf y \in S(\mathbb R^{n+1})$.
Let
\[
u = (\max d_i) \mu(\mathbf f,\mathbf g) \rho(\mathbf x,\mathbf y)
\text{ and } 
v = \mu(\mathbf f,\mathbf x) \|\mathbf f-\mathbf g\|
.\]
Then,
\[
\frac{1}{1+u+v} \mu(\mathbf f,\mathbf x)
\le
\mu(\mathbf g,\mathbf y) 
\le
\frac{1}{1-u-v} \mu(\mathbf f,\mathbf x)
.
\]
\end{theorem}

\begin{remark}Similar formulas appeared in~\ocite{Buergisser-Cucker} and
~\ocite{DMS}. The final form here appeared in~\ocite{NONLINEAR-EQUATIONS}
and generalizes to the sparse condition number.
\end{remark}

\begin{proof}
Let $R$ be a rotation taking $\mathbf y$ to $\mathbf x$. Then,
$\mu(\mathbf g,\mathbf y) = \mu(\mathbf g \circ R, \mathbf x)$. Moreover,
it is easy to check that $\|\mathbf g \circ R - \mathbf g\| \le (\max d_i) \rho(\mathbf x,\mathbf y)$.
Thus,
\[
\mu(\mathbf f,\mathbf x) \|\mathbf f - \mathbf g \circ R\| \le (u+v) .
\]

Now, notice that Proposition~\ref{prop:mu}(3) implies:
\[
\frac{1}{\mu(\mathbf f,\mathbf x)} - 
\|\mathbf f - \mathbf g \circ R\|.
\le
\frac{1}{\mu(\mathbf g \circ R,\mathbf x)} \le 
\frac{1}{\mu(\mathbf f,\mathbf x)} + 
\|\mathbf f - \mathbf g \circ R\|.
\]

The theorem follows by 
taking inverses.
\end{proof}

\section{The inclusion theorem}

For any $\mathbf x \in S(\HdR)$, we denote by $A_{\mathbf x}$ be the affine space 
$\mathbf x + \mathbf x^\perp$ and by $F_\mathbf x: A_{\mathbf x} \rightarrow \mathbb R^n$,
$\mathbf X \mapsto \mathbf f(\mathbf x+\mathbf X)$ the
restriction of $\mathbf f$ to $A_{\mathbf x}$. Then $F_{\mathbf x}$ is an $n$-variate
polynomial system of degree $\mathbf d$.

\begin{lemma}\label{lem:gamma:mu}\cite{Bezout1}\label{high-derivatives}
\[
 \gamma(\mathbf F_\mathbf x, 0) \le 
\frac{(\max d_i)^{3/2}}{2} \| \mathbf f\| \mu(\mathbf f,\mathbf x)
\]
\end{lemma}
\begin{proof}
For simplicity assume $\|\mathbf f\|=1$.
Let $k \ge 2$ and 
\[
\Delta = \left[\begin{matrix}\sqrt{d_1} \\ & \ddots \\
& & \sqrt{d_n} \end{matrix} \right]. 
\]
\begin{eqnarray*}
\frac{1}{k!}
\left\| D\mathbf F_\mathbf x(0)^{-1} D^k\mathbf F_\mathbf x(0) \right\|
&=&
\frac{1}{k!}
\left\| D\mathbf f(\mathbf x)_{|\mathbf x^\perp}^{-1} D^k\mathbf f(\mathbf x)_{|\mathbf x^\perp} \right\|
\\
&\le&
\frac{1}{k!}
\left\| D\mathbf f(\mathbf x)_{|\mathbf x^\perp}^{-1} \Delta\right\| \left\|\Delta^{-1} D^k\mathbf f(\mathbf x)_{|\mathbf x^\perp} \right\|
\\
&\le&
\mu(\mathbf f,\mathbf x)
\
\frac{1}{k!}
\left\| \Delta^{-1} D^k\mathbf f(\mathbf x)_{|\mathbf x^\perp} \right\|
\\
\end{eqnarray*}

Now, notice that
\[
\begin{split}
|D^k \mathbf f_i(\mathbf x)| = |\langle \mathbf f_i, D^k K_{d_i}(\cdot, \mathbf x) \rangle|
&\le\\
\le
\|\mathbf f_i\| \sup_{
\substack{
\|\mathbf u_1\|=\cdots = \|\mathbf u_k\|=1\\
\mathbf u_1, \dots, \mathbf u_k \perp \mathbf x
}}& 
\|D^k K_{d_i}(\cdot, \mathbf x)(\mathbf u_1, \dots, \mathbf u_k)\|
\end{split}\]
where $K_{d_i}(\mathbf y,\mathbf x) = \langle \mathbf y,\mathbf x \rangle^{d_i}$ is the reproducing kernel of $\mathcal H_{d_i}^{\mathbb R}$.
Differentiating $K_{d_i}$ with respect to $\mathbf y$, one obtains:
\[
\frac{1}{k!}
D^kK_{d_i}(\mathbf y, \mathbf x)(\mathbf u_1, \dots, \mathbf u_k)
 = 
\binomial{d_i}{k} \langle \mathbf y,\mathbf x \rangle^{d-k}
\langle y,\mathbf u_1 \rangle \cdots \langle y, \mathbf u_k \rangle
.
\]
The norm of $\frac{1}{k!} D^kK_{d_i}(\mathbf y,\mathbf x)(\mathbf u_1, \dots, \mathbf u_k)$ (as a polynomial of $\mathbf y$) can
be computed using the reproducing kernel property.
\begin{eqnarray*}
\left\|
\frac{1}{k!} D^kK_{d_i} (\cdot, \mathbf x)(\mathbf u_1, \dots, \mathbf u_k)
\right\|^2
=\hspace{-10em}&& \\
&=&
\left\langle
\frac{1}{k!} D^kK_{d_i} (\cdot, \mathbf x)(\mathbf u_1, \dots, \mathbf u_k),
\frac{1}{k!} D^kK_{d_i} (\cdot, \mathbf x)(\mathbf u_1, \dots, \mathbf u_k)
\right\rangle
\\
&=&
\frac{1}{k!}
\frac{\partial \mathbf y}{\partial \mathbf u_1} 
\cdots
\frac{\partial \mathbf y}{\partial \mathbf u_k} 
\binomial{d_i}{k} 
\langle \mathbf y,\mathbf x \rangle^{d-k}
\langle \mathbf y,\mathbf u_1 \rangle \cdots \langle \mathbf y, \mathbf u_k \rangle
\\
&=&
\frac{1}{k!}
\binomial{d_i}{k} 
\mathrm{Perm}
\left[
\begin{matrix}
\langle \mathbf u_i, \mathbf u_j \rangle
\end{matrix}
\right]
\\
&\le&
\binomial{d_i}{k} 
\end{eqnarray*}

It follows that
\[
\frac{1}{k!}
\left\| D\mathbf F_\mathbf x(0)^{-1} D^k\mathbf F_\mathbf x(0) \right\|
\le
\mu(\mathbf f,\mathbf x) \max \frac{1}{\sqrt{d_i}} \binomial{d_i}{k}
.
\]

Estimating $\binomial{d_i}{k} \le d_i^k 2^{-k}$ and
using Exercise~\ref{ex:mu}, 

\[
\gamma(\mathbf F_\mathbf x,0) \le \frac{d^{3/2}}{2} \mu(\mathbf f,\mathbf x) 
.
\]
\end{proof}

Whenever the sequence $(\mathbf X_k)_{k \in \mathbb N}$ defined by
$\mathbf X_0=0$, $\mathbf X_{k+1} =$\linebreak$=N(\mathbf F_\mathbf x, \mathbf X_{k})$ converges,
let $\mathbf X^* = \lim \mathbf X_k$ and define
\[
\zeta_x = \frac{ \mathbf x + \mathbf X^*}{\| \mathbf x+\mathbf X^{*} \|} \in S^{n+1}
.
\]

As in Theorem~\ref{th:alpha}, define
\[
r_0(\alpha) = 
\frac{1+\alpha-\sqrt{1-6\alpha+\alpha^2}}{4\alpha} 
\]
Let $\alpha_*$ the smallest positive root of
\[
\alpha_* = \alpha_0 (1 - \alpha_* r_0(\alpha_*))^2 .
\]
Numerically, $\alpha_* > 0.116$. (This is better than~\cite{CKMW1}).
Let $B_\mathbf x = \{\mathbf y \in S^n: \rho(\mathbf x,\mathbf y) \le r_{\mathbf x}\}$
with $r_\mathbf x = r_0(\alpha_*) \mu(\mathbf f,\mathbf x) \|\mathbf f(\mathbf x)\| $.

\begin{theorem}\label{th:inclusion}
Let $\mathbf f \in S(\HdR)$ and $\mathbf x \in S^n$ be such that
\[
(\max d_i)^{3/2} \mu(\mathbf f,\mathbf x)^2 \|\mathbf f(\mathbf x)\| \le \alpha_* .
\]

Then,
\begin{enumerate}
\item
$\alpha(\mathbf F,0) \le \alpha_*$. 
\item
$0$ is an approximate zero of the second kind of $\mathbf F_\mathbf x$,
and in particular $\mathbf f(\zeta_x) = 0$. 
\item
$\zeta_\mathbf x \in B_\mathbf x$.
\item
For any $\mathbf z \in B_\mathbf x$, $\zeta_\mathbf z = \zeta_\mathbf x$.
\end{enumerate}
\end{theorem}

\begin{proof}
\begin{enumerate}
\item By Lemma~\ref{high-derivatives}, 
\[
\begin{split}
\alpha(\mathbf F_\mathbf x,0) \le 
(\max d_i)^{3/2} \mu(\mathbf f,\mathbf x) 
\left\| D\mathbf f(\mathbf x)_{\mathbf x^{\perp}}^{-1} \mathbf f(\mathbf x) \right\|
\le \hspace{3em}\\
\hspace{3em}\le
(\max d_i)^{3/2} \mu(\mathbf f,\mathbf x)^2 \|\mathbf f(\mathbf x)\| 
\le
\alpha_*
.
\end{split}
\]
\item
Since $\alpha_* \le \alpha$, we can apply 
Theorem~\ref{th:alpha} to $\mathbf F_\mathbf x$ and $0$.
\item Since $0$ is a zero of the second kind
for $\mathbf F_\mathbf x$, 
\[\mathbf F_\mathbf x(\mathbf X^*) = \mathbf f( \|\mathbf x+\mathbf X^*\| \zeta_\mathbf x) = 0\]
and hence by homogeneity $\mathbf f(\zeta_\mathbf x) = 0$.
\item
\[
\rho(\mathbf x,\zeta_\mathbf x)\le \tan 
\rho(\mathbf x,\zeta_\mathbf x)\le 
r_0(\alpha_*) \beta(\mathbf f,\mathbf x) \le 
r_0(\alpha_*) \mu(\mathbf f,\mathbf x) \|\mathbf f(\mathbf x)\| 
\]
\item
By Theorem~\ref{th:varmu},
\[
\mu(\mathbf f,\mathbf z) \le 
\frac{1}{1- (\max d_i) \mu(\mathbf f,\mathbf x) \rho(\mathbf x,\mathbf z)} \mu(\mathbf f,\mathbf x)
\le
\frac{1}{1- \alpha^* r_0(\alpha_*)} \mu(\mathbf f,\mathbf x)
\]
and hence, as in item 1:
\[
\alpha(\mathbf F_\mathbf z,0) \le
\frac{1}{(1- \alpha^* r_0(\alpha_*))^2} \alpha_* \le \alpha_0
.
\]
\end{enumerate}
\end{proof}
This theorem appeared in~\ocite{CKMW1}. For other inclusion/exclusion
theorems based in alpha-theory, see~\ocite{GLSY}.

\section{The exclusion lemma}

\begin{lemma}\label{lem:exclusion}
Let $\mathbf f \in S(\HdR)$ and let $\mathbf x,\mathbf y \in S^n$ with
$\rho(\mathbf x,\mathbf y) \le \sqrt{2}$. Then,
\[
\|\mathbf f(\mathbf x)-\mathbf f(\mathbf y)\| \le 
\max(d_i) \rho(\mathbf x,\mathbf y). 
\]
In particular, let $\delta = 
\min( \|\mathbf f(\mathbf x)\| / \sqrt{\max(d_i)}, \sqrt{2} )$.
If $\mathbf f(\mathbf x) \ne 0$, then there is no zero
of $\mathbf f$ in 
\[
B(\mathbf x, \delta) = \{ \mathbf y \in S^{n+1}: \rho(\mathbf x,\mathbf y) \le \delta \}
.\]
\end{lemma}

\begin{proof}
First of all, 
\begin{eqnarray*}
|f_i(x) - f_i(y)| &=& |\langle f_i(\cdot),
K_{d_i}(\cdot,\mathbf x) - K_{d_i}(\cdot, \mathbf y)\rangle|
\\
&\le&
\|f_i\| 
\| K_{d_i}(\cdot,\mathbf x) - K_{d_i}(\cdot, \mathbf y)\|
\\
&\le&
\|f_i\| 
\sqrt{ K_{d_i}(\mathbf x,\mathbf x) + K_{d_i}(\mathbf y,\mathbf y) - 2 K_{d_i}(\mathbf x,\mathbf y) }
\\
&=&
\|f_i\| \sqrt{2}
\sqrt{1 - \cos(\theta)^d}
\end{eqnarray*}
with $\theta = \rho(x,y)$.
Since $\theta \le \pi < \sqrt{30}$, we have always
\[
\cos(\theta) = 1 - \frac{1}{2} \theta^2 
+ \frac{1}{4!} \theta^4 
- \frac{1}{6!} \theta^6 
+ \cdots >  
1 - \frac{1}{2}\theta^2.
\]

The reader will check that for $\epsilon < 1$, $(1-\epsilon)^d > 1 - d \epsilon$. Therefore, using $\theta < 1/\sqrt{2}$,
\[
|f_i(\mathbf x) - f_i(\mathbf y)| \le 
\|f_i\| \sqrt{d_i}\theta
\]
and
\[
\|\mathbf f(\mathbf x) - \mathbf f(\mathbf y)\| \le 
\sqrt{\max(d_i)}\theta
.
\]
\end{proof}

\part{The algorithm and its complexity}

\section{Convexity and geometry Lemmas}
\begin{definition} Let $\mathbf y_1, \dots, \mathbf y_s \in S^n$ belong to the same
hemisphere, that is $\langle \mathbf y_i, \mathbf z\rangle > 0$ for a fixed $\mathbf z$.
The {\bf spherical convex hull} of $\mathbf y_1, \dots, \mathbf y_s$ is defined
as
\[
\begin{split}
\mathrm{SCH}(\mathbf y_1, \dots, \mathbf y_s) = 
\left\{ \rule{0em}{4ex}
\frac
{\lambda_1 \mathbf y_1 + \cdots + \lambda_s \mathbf y_s}
{\|\lambda_1 \mathbf y_1 + \cdots + \lambda_s \mathbf y_s\|}
: \lambda_1, \dots, \lambda_s \ge 0 \right. \hspace{2em}\\\hspace{2em}
\left. \text{ and }
 \lambda_1 + \cdots + \lambda_s = 1 
\rule{0em}{4ex}
\right\}.
\end{split}
\]
\end{definition}

This is the same as the intersection of the sphere with the cone
$\{ \lambda_1 \mathbf y_1 + \cdots + \lambda_s \mathbf y_s: \lambda_1, \dots, \lambda_s \ge 0\}$.
We will need the following convexity Lemma from~\ocite{CKMW1}:

\begin{lemma}\label{lem:convexity}
Let $\mathbf y_1, \dots, \mathbf y_s \in S^n$ belong to the same hemisphere.
Let $r_1, \dots, r_s > 0$ and let $B(\mathbf y_i,r_i)=\{\mathbf x \in S^n: \rho(x, \mathbf y_i)<r_i\}$.
If $\cap B(\mathbf y_i, r_i) \ne \emptyset$, then 
$\mathrm{SCH}(\mathbf y_1, \dots, \mathbf y_s) \subset \cup B(\mathbf y_i, r_i)$. 
\end{lemma}

\begin{exercise}
Prove Lemma~\ref{lem:convexity} above.
\end{exercise}

For the root counting algorithm, we will need to define a {\bf mesh}
on the sphere. 

\begin{lemma}\label{lem:mesh}
For every $\eta =2^{-t}$, we can construct a set $C(\eta) \subseteq S^n$ 
satisfying:
\begin{enumerate}
\item For all $\mathbf z \in S^n, \exists \mathbf x \in C(\eta)$ such that
$\rho(\mathbf z,\mathbf x) \le \eta\sqrt{n}/2$.
\item For all $\mathbf x \in S^n$, let $Y = \{ \mathbf y \in C(\eta): \rho(\mathbf x,\mathbf y) \le \sqrt{n}\eta\}$.
Then $\mathbf x \in \mathrm{SCH}(Y)$.
\item $\# C(\eta) \le 2n(1+2^{t+1})^n $.
\end{enumerate}
\end{lemma}

\begin{proof}
Just set
\[
C(\eta) = 
\left \{ \frac{\mathbf x}{\|\mathbf x\|}: \mathbf x \in \mathbb R^{n+1}, x_i \eta^{-1} \in \mathbb Z,
\|\mathbf x\|_{\infty}=1 \right \}
.
\]
This corresponds to dividing $Q=\{\mathbf x: \|\mathbf x\|_{\infty}=1\}$ into 
$n$-cubes of side $\tilde \eta$. The maximal distance in $Q$
between a point $\mathbf Z \in Q$ and a point $\mathbf X$ in the mesh is half of the diagonal,
or $\eta \sqrt{n}$. Then
\[
\rho(\mathbf Z/\|\mathbf Z\|, \mathbf X/\|\mathbf X\|) < \eta \sqrt{n}.
\]

Now, let $Y'$ be the set of points $\mathbf y \in C(\eta)$ such that
the distance along $Q$ between
$\mathbf x/\|\mathbf x\|_{\infty}$ and $\mathbf y/\|\mathbf y\|_{\infty}$ is at most $\eta$.
Then clearly $\mathbf x \in \mathrm{SCH}(Y')$. Moreover,
$Y'\subset Y$.  

The last item is trivial.
\end{proof}

\section{The counting algorithm}
Given $\mathbf f \in S(\HdR)$ and
$\eta=2^{-t}$, we construct a graph $\mathcal G_{\eta}=(\mathcal V_\eta, \mathcal E_\eta)$ as follows.
Let
\[
A(\mathbf f) = \{ \mathbf x \in S^n: \max d_i^{3/2} \mu(\mathbf f,\mathbf x)^2 \|\mathbf f(\mathbf x)\| < \alpha_*\}
\]
be the set of points satisfying the hypotheses of Theorem~\ref{th:inclusion}.
The set of vertices of $\mathcal G_{\eta}$ is $\mathcal V_{\eta} = C(\eta) \cap A(\mathbf f)$.

Recall that Let $B_\mathbf x = \{\mathbf y \in S^n: \rho(\mathbf x,\mathbf y) \le r_{\mathbf x}\}$
with $r_\mathbf x = r_0(\alpha_*) \mu(\mathbf f,\mathbf x) \|\mathbf f(\mathbf x)\| $.
The set of edges of $\mathcal G_{\eta}$ is $\mathcal E_\eta = \{(\mathbf x,\mathbf y) \in \mathcal V_\eta \times \mathcal V_\eta:
B_\mathbf x \cap B_\mathbf y \ne 
\emptyset \}$. This graph is clearly constructible. Theorem~\ref{th:inclusion}
implies that for any edge $(\mathbf x,\mathbf y)\in \mathcal E_\eta$, $\zeta_\mathbf x = \zeta_\mathbf y$. More
generally,

\begin{lemma}
The vertices of any connected component of $\mathcal G(\eta)$
are approximate zeros associated to the same zero of $\mathbf f$.
Moreover, if $\mathbf x,\mathbf y$ belong to distinct connected components
of $\mathcal G(\eta)$, then $\zeta_\mathbf x \ne \zeta_\mathbf y$.
\end{lemma}

The algorithm is as follows:

{\tt
\begin{trivlist}
\item {\bf Algorithm} RootCount
\item {\bf Input:} $\mathbf f \in S(\HdR)$.
\item {\bf Output:} $\# \zeta \in S^n: \mathbf f(\zeta) = 0$.
\item
\item $\eta \leftarrow 2^{- \lceil \log_2 (1/\sqrt{2n}) \rceil }$.
\item {\bf Repeat}
\subitem $\eta \leftarrow \eta/2$.
\subitem Let $\mathcal U_1, \dots, \mathcal U_r$ be the connected components 
of $\mathcal G_{\eta}$.
\subitem {\bf Until} 
$\forall 1 \le i < j \le r, \forall \mathbf x \text { vertex of } \mathcal U_i, 
\forall \mathbf y \text{ vertex of }\mathcal U_j,$
\begin{equation}\label{eq:separation}
\rho(\mathbf x,\mathbf y) > 2 \eta \sqrt{n}.
\end{equation}
\subitem{\bf and} $\forall \mathbf x \in C(\eta) \setminus A(\mathbf f)$,
\begin{equation}\label{eq:exclusion}
\|\mathbf f(\mathbf x)\| > \eta \sqrt{n \max d_i}/2 .
\end{equation}
\item {\bf Return} $r$.
\end{trivlist}
}

\begin{theorem}\label{th:correctness}
If the algorithm {\tt RootCount} stops, then
$r$ is the correct number of roots of $\mathbf f$ in $S^n$.
\end{theorem}

\begin{proof}[Proof of Th.\ref{th:correctness}]

Suppose the algorithm stopped at a certain value of $\eta$.
As each connected component $\mathcal U_i$ determines a distinct and
unique zero of $\mathbf f$, it remains to prove that there are no
zeros of $\mathbf f$ outside $\cup_{\mathbf x \in \mathcal V_\eta} B_\mathbf x$.

Therefore, assume by contradiction that there is $\zeta \in S^n$
with $\mathbf f(\zeta) = 0$ and $\zeta \not \in B_\mathbf x$ for any 
$\mathbf x \in V_{\eta}$. 

Let $Y$ be the set of $\mathbf y \in C(\eta)$ with 
$\rho(\zeta,\mathbf y) \le \eta \sqrt{n} .$ 

If there is $\mathbf y \in Y$ with $\mathbf y \not \in A(\mathbf f)$  
let $\delta = \|\mathbf f(\mathbf y)\|/\sqrt{\max d_i}$.
Equation 
\eqref{eq:exclusion}
guarantees that $\eta \sqrt{n}/2 < \delta$. By construction,
$\eta \sqrt{n}/2 < \sqrt{2}$. Therefore, the exclusion lemma
~\ref{lem:exclusion} guarantees that $\mathbf f(\zeta) \ne 0$, contradiction.

Therefore, we assume that $Y \subset A(\mathbf f)$. Equation~\eqref{eq:separation}
guarantees that $Y \subset \mathcal U_k$ for a same connected component of $\mathcal G_{\eta}$.
Therefore, $\cap_{\mathbf y \in Y} B_\mathbf y \ni \zeta$ is not empty.

By Lemma~\ref{lem:mesh}(2), $\mathbf x \in \mathrm{SCH}(Y)$.
Lemma~\ref{lem:convexity} says that
\[
\mathrm{SCH}(Y) \subseteq \cup_{\mathbf y \in Y} B_\mathbf y
\]

Thus, $\mathbf x \in B_\mathbf y$ for some $\mathbf y$, contradiction again.

\end{proof}

A consequence of Th.\ref{th:correctness} is that {\bf if the algorithm
stops}, one can obtain an
approximate zeros of the second kind for 
each root of $f$ by recovering one vertex for each connected
component.

\section{Complexity}

We did not prove that algorithm {\tt RootCount} stops. It actually
stops {\bf almost surely}, that is for input $f$ outside a certain
measure zero set.

Define
\[
\kappa(\mathbf f,\mathbf x) = \frac{1}{\sqrt{\mu(\mathbf f,\mathbf x)^{-2} + \|\mathbf f(\mathbf x)\|^2}}
\]
and notice that
\[
\kappa(\mathbf f,\mathbf x) \le \mu(\mathbf f,\mathbf x)
\text{ and }
\kappa(\mathbf f,\mathbf x) \le \|\mathbf f(\mathbf x)\|^{-1}
.
\]
Reciprocally,
\[
\min( \mu(\mathbf f,\mathbf x), \|\mathbf f(\mathbf x)\|^{-1} ) \le \sqrt{2} \kappa(\mathbf f,\mathbf x)
.
\]
If $\mathbf f(\mathbf x)=0$, then $\kappa(\mathbf f,\mathbf x) = \mu(\mathbf f,\mathbf x)$.

\begin{definition}
The {\bf condition number} for for Problem~\ref{pr:sphere} (counting real zeros 
on the sphere) is
\[
\kappa(\mathbf f) = \max_{\mathbf x \in S^n} \kappa(\mathbf f,\mathbf x).
\]
\end{definition}

Assume that $\mathbf f$ has no degenerate root. Then the denominator
is bounded away from zero, and $\kappa(\mathbf f)$ is finite.
We will prove later that the algorithm stops for $\kappa(\mathbf f)$ finite.
But before, we state and prove the {\bf condition number theorem} to
obtain some geometric intuition on $\kappa(\mathbf f)$.

\begin{theorem}\fullcite{CKMW2}
Let $\Sigma^{\mathbb R} = \{ \mathbf g \in \HdR: \exists \zeta \in S^n: \mathbf g(\zeta)=0 \text{ and }
\mathrm{rk}(D\mathbf g(\zeta))<n\}$. Let $\mathbf f \in S(\HdR)$, $\mathbf f \not \in \Sigma^{\mathbb R}$.
Then,
\[
\kappa(\mathbf f) = \frac{1}{\min_{\mathbf g \in \Sigma^{\mathbb R}} \|\mathbf f-\mathbf g\|} .
\]
In particular, $\kappa(\mathbf f) \ge 1$.
\end{theorem}

\begin{proof}
It suffices to prove that
\[
\kappa(\mathbf f,\mathbf x) = 
\frac{1}{\min_{
\substack{\mathbf g \in \HdR\\\mathbf g(\mathbf x)=0\\\mathrm{rk}(D\mathbf g(\mathbf x))<n}} \|\mathbf f-\mathbf g\|} .
\]
We proceed as in the proof of Prop.\ref{prop:mu}.
We decompose
\[
\HdR
=
H_0 \oplus H_1 \oplus H_2 \oplus \cdots
\]
where $H_0$ and $H_1$ correspond to the constant and
linear terms of $\mathbf y \mapsto \mathbf f(\mathbf x+\mathbf y)$. Let $\mathbf u_1, \dots, \mathbf u_n$
be an orthonormal basis for $\mathbf x^\perp$.

An orthonormal basis for $H_0 \oplus H_1$ is
\[
\left( K_{d_i}(\cdot,\mathbf x), 
\frac{1}{\sqrt{d}} \frac{\partial K_{d_i}(\cdot, \mathbf x)}{\partial \mathbf u_j} \right)
.
\]

The projection of $\mathbf f$ in $H_0 \oplus H_1$ is
\begin{gather*}
\left[
\begin{matrix}
\langle \mathbf f(\cdot), K_{d_i}(\cdot, \mathbf x) \rangle
\end{matrix}
\right]
 \oplus
\left[
\begin{matrix}
& \vdots & \\
\cdots & 
\left\langle \mathbf f_i,  \frac{1}{\sqrt{d}} \frac{\partial K_{d_i}(\cdot, \mathbf x)}{\partial \mathbf u_j} \right \rangle
& \cdots\\
& \vdots & 
\end{matrix}
\right]
=\hspace{6em}\\
\hspace{6em}=
\mathbf f(\mathbf x) \oplus
\left[
\begin{matrix}
d_1^{-1/2}\\
&d_2^{-1/2}\\
&&d_n^{-1/2}\\
\end{matrix}
\right]
D\mathbf f(\mathbf x)_{|\mathbf x^{\perp}}
.
\end{gather*}
This is an orthogonal projection onto $\mathbb R^n \times \mathbb R^{n \times n}$.

Now,
\[
\kappa(\mathbf f,\mathbf x)^{-2} = 
\|\mathbf f(\mathbf x)\|^2 + 
\sigma_n
\left(
\left[
\begin{matrix}
d_1^{-1/2}\\
&d_2^{-1/2}\\
&&d_n^{-1/2}\\
\end{matrix}
\right]
D\mathbf f(\mathbf x)_{|\mathbf x^{\perp}}
\right)
.
\]
  
Again, we apply Th.\ref{th:eckart-young}.
\end{proof}

\medskip
\par

\begin{lemma}\label{lem:sep:kappa}
Let $\zeta_1, \zeta_2$ be distinct roots of $\mathbf f$ in $S^n$.
Then,
\[
\rho(\zeta_1, \zeta_2) \ge 
\frac{1}
{\max d_i^{3/2} \kappa(\mathbf f) }
\]
\end{lemma}

\begin{proof}
\begin{align*}
\| \zeta_1 - \zeta_2 \| &\ge \frac{1}{2 \gamma(\mathbf f,\zeta_1)}
&&\text{by Ex.\ref{ex:separation}}
\\
&\ge
\frac{1}
{\max d_i^{3/2} \mu(\mathbf f,\zeta_1) } 
&&\text{by Lem.\ref{lem:gamma:mu}}
\\
&\ge
\frac{1}
{\max d_i^{3/2} \kappa(\mathbf f) }
&&\text{because $\mathbf f(\zeta_1)=0$.}
\end{align*}
The Lemma follows.
\end{proof}

\begin{lemma} Assume that 
\[ 
\eta < 
\frac{1}{2 \max d_i^{3/2} \sqrt{n}\kappa(\mathbf f)}
(1-2 \alpha_* r_0(\alpha_*))
.
\]
Then \eqref{eq:separation} holds.
\end{lemma}

\begin{proof}
Recall that $\mathbf x$ and $\mathbf y$ belong to $A_\mathbf f$, so that
\[
\max d_i^{3/2} \mu(\mathbf f,\mathbf x)^2 \|\mathbf f(\mathbf x)\| < \alpha_* 
\]
and the same for $\mathbf y$. In particular, the radius $r_\mathbf x$ of $B_\mathbf x$
satisfies
\[
r_0(\alpha_*) \mu(\mathbf f,\mathbf x) \|\mathbf f(\mathbf x)\| <
\frac{ \alpha_* r_0(\alpha_*) }
{\max d_i^{3/2} \mu(\mathbf f,\mathbf x)}
\le
\frac{ \alpha_* r_0(\alpha_*) }
{\max d_i^{3/2} \kappa(\mathbf f,\mathbf x)}
.
\]

By Lemma~\ref{lem:sep:kappa} and the triangle inequality,
\begin{eqnarray*}
\rho(\mathbf x,\mathbf y) &\ge&
\rho(\zeta_\mathbf x, \zeta_\mathbf y) - 
r_0(\alpha_*) \mu(\mathbf f,\mathbf x) \|\mathbf f(\mathbf x)\| -
r_0(\alpha_*) \mu(\mathbf f,\mathbf y) \|\mathbf f(\mathbf y)\| 
\\
&\ge&
\frac{1}{\max d_i^{3/2} \kappa(\mathbf f)}
(1-2 \alpha_* r_0(\alpha_*))
.
\end{eqnarray*}
\end{proof}

\begin{lemma}
Let $\mathbf x \not \in A_f$. Then,
\[
\|\mathbf f(\mathbf x)\| \ge 
\frac{ \alpha_* }{\kappa(\mathbf f,\mathbf x)^2 \max d_i^{3/2}} 
.
\]
\end{lemma}

\begin{proof}
Let $\mathbf x \not \in A_\mathbf f$, so that
\[
\frac{\max d_i^{3/2}}{2} \mu(\mathbf f,\mathbf x)^2 \|\mathbf f(\mathbf x)\| \ge \alpha_* .
\]
Recall that
\[
\min( \mu(\mathbf f,\mathbf x), \|\mathbf f(\mathbf x)\|^{-1} ) \le \sqrt{2} \kappa(\mathbf f,\mathbf x)
\]
There are two possibilities.
If $\mu(\mathbf f,\mathbf x) \le \sqrt{2}\kappa(\mathbf f,\mathbf x)$, then 

\[
\|\mathbf f(\mathbf x) \| \ge \frac{\alpha_*}{\max d_i^{3/2} \kappa(\mathbf f,\mathbf x)^2} .
\]

Otherwise, 
\[
\|\mathbf f(\mathbf x)\| \ge \frac{1}{\sqrt{2} \kappa(\mathbf f,\mathbf x)}
\ge
\frac{\alpha_*}{\max d_i^{3/2} \kappa(\mathbf f,\mathbf x)^2} .
\]

\end{proof}

Now we can state the `cloud complexity' theorem.

\begin{theorem}
The algorithm {\tt RootCount} will stop for
\[
\eta < 
\frac{1}{\max d_i^{3/2} \kappa(\mathbf f)^2} 
\min 
\left( \alpha_*\ ,\ 
\frac{\kappa(\mathbf f)} {2\sqrt{n}}
(1-2 \alpha_* r_0(\alpha_*))
\right)
\] that is, after $O(\log \kappa(\mathbf f) + \log \max d_i)$ iterations.
The total number of evaluations of $\mathbf f$ and $D\mathbf f$ is
\[
2n (1 + 4 \max d_i^{3/2} \sqrt{n} \kappa(\mathbf f)^2)^n 
.
\]
\end{theorem}

That means that 
$2n (1 +4 \max d_i^{3/2} \sqrt{n} \kappa(\mathbf f)^2)^n$ processors
in parallel can compute the root count in time 
$O(\log \kappa(\mathbf f) + \log \max d_i)$
times a polynomial in $n$ for the linear algebra.

For people concerned with the overall computing cost, a 
price tag exponential in $n$ is known as the {\bf curse of
dimensionality}. It usually plagues divide and conquer 
and Monte-Carlo algorithms. 

But the situation $n=2$ is already interesting. How efficiently can 
we count
zeros of a system of polynomials on the $2$-sphere? 
As the parallel and sequential running time depends upon $\kappa(f)$,
it is useful to known more about the condition number.

\section{Probabilistic and smoothed analysis}

One possibility is to pick the input system $\mathbf f$ at random, and
treat $\kappa(\mathbf f)$ as a random variable. For instance, let
$\mathbf f \in \HdR$ be random with {\bf Gaussian} probability distribution
\[
\frac{1}{(2 \pi)^{\dim \HdR / 2}}
e^{-\|f\|^2/2}
\ \mathrm{d}\HdR
.
\]

The tail for the random variable $\kappa(\mathbf f)$ and
the expected value of $\log \kappa(\mathbf f)$ can be bounded by

\begin{theorem}\label{boundkapa}\fullcite{CKMW3}
\linebreak
Let $\mathbf f$ be as above. Assume that $n\geq 3$. Then,
\begin{trivlist}
\item[(i)]
For $a > 4\sqrt 2\,(\max d_i)^2n^{7/2}N^{1/2}$ we have 
$$
  \mathrm {Prob} \big( \kappa (\mathbf f)>a \big)\leq K_n \frac{\sqrt{2n}(1+\ln
  (a/\sqrt{2n}))^{1/2}}{a},
$$
where $N=\dim \HdR$,
$K_n:=8{(\max d_i)}^2{\mathcal{D}}^{1/2}\,{N}^{1/2}n^{5/2}+1$ 
and $\mathcal D = \prod d_i$.
\item[(ii)]
$$
  \mathbb E (\ln \kappa (\mathbf f))\leq \ln K_n + (\ln K_n)^{1/2}+(\ln
  K_n)^{-1/2} + \frac{1}{2}\ln(2n).
$$
\end{trivlist}
\end{theorem}

Notice as a consequence that the expected running time
of {\tt RootCount} is 
$\mathbb E (\ln \kappa (\mathbf f)) \in \mathcal O(n\ln \max d_i)$.
This is cloud computing time, of course.

\medskip
\par
Average time analysis depends upon an arbitrary distribution.
\ocite{Spielman-Teng} suggested looking instead at a small random
perturbation for each given input. This 
is known as {\bf smoothed analysis}. 
\par
For a given $\mathbf f \in S(\HdR)$, we will consider 
the uniform distribution in the ball $B(\mathbf f, \arcsin \sigma) \subset
S(\HdR)$ where $\sigma$ is an arbitrary radius, and
Riemannian metric on the sphere is assumed. The strange looking
arcsine comes from the fact that $B(\mathbf f, \arcsin \sigma)$ is
the projection on the sphere of the ball $B(\mathbf f,\sigma) \subset \HdR$.
The reason for looking at the uniform distribution for perturbations
instead of Gaussian is the following result:

\begin{theorem}\fullcite{BCL}
Let $\Sigma \subset \mathbb R^N$ be contained
in a projective hypersurface $H$ of degree at
most $D$ and let $\kappa: \mathbb S^{N-1} \rightarrow [1, \infty]$
be given by
\[
\kappa(\mathbf f) = \frac{\|\mathbf f\|}{\min_{\mathbf g \in \Sigma} \|\mathbf f - \mathbf g\|}
.
\]
 
Then, for all $\sigma \in (0,1]$,
\[
\sup_{\mathbf f \in S^{N-1}} \mathrm{E}_{\mathbf h \in B(\mathbf f,\arcsin \sigma) 
\subseteq S^{N-1}}
(\ln \kappa(\mathbf h)) \le 2 \ln (N-1) + 2 \ln D - \ln \sigma + 5.5 .
\]
\end{theorem}

In the context of the root counting problem, the degree $D$ of $\Sigma = \Sigma^{\mathbb R}$ 
is bounded by $n^2 (\prod d_i)(\max d_i)$. Therefore,

\begin{corollary}\fullcite{CKMW2}\label{cor:smoothed}
\[
\begin{split}
\sup_{\mathbf f \in S(\HdR)} \mathrm{E}_{\mathbf h \in B(\mathbf f,\arcsin \sigma) 
\subseteq S(\HdR)}
(\ln \kappa(h)) 
\le 2 \ln (\dim(\HdR)) + 
4 \ln(n) \\
+ 2 \ln (\prod d_i) + \ln 1/\sigma + 6 .
\end{split}
\]
\end{corollary}

\addtocontents{toc}{\vspace{2ex}}
\section{Conclusions}

We sketched the average time analysis and a smoothed analysis 
of an algorithm for real root counting and, incidentally, root finding. 
The same algorithm
can also decide if a given polynomial system admits a root.

Loosely speaking, deciding (resp. counting) roots of polynomial systems
are NP-complete (resp.  $\#$P complete)
problems. The formal NP-complete and $\#$P-complete problems refer
to {\bf sparse} polynomial systems.

Our algorithm requires actually polynomial evaluations, so it can
take advantage of the sparse structure. Moreover,
the degree of the sparse
discriminant is no more than the degree of the usual discriminant.
In that sense Corollary~\ref{cor:smoothed} 
is still valid. The running time of the algorithm is 
polynomial in $n$ and in the dimension of the input space.
Again, this is a massively parallel algorithm so the number of
processors is exponential in $n$.

\end{document}

%% file: images/psi.pdf_t
\begin{picture}(0,0)%
\includegraphics{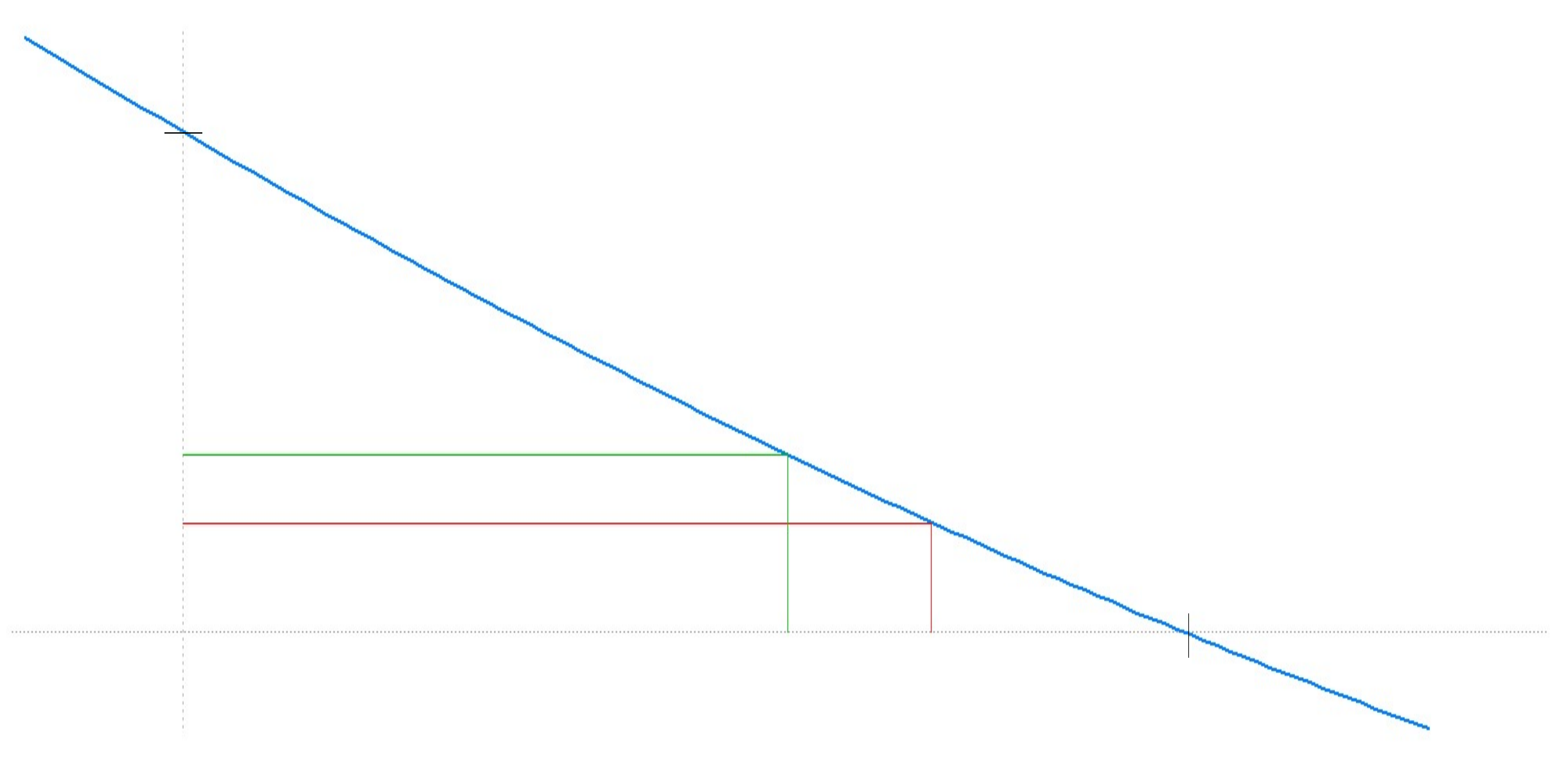}%
\end{picture}%
\setlength{\unitlength}{4144sp}%
\begingroup\makeatletter\ifx\SetFigFont\undefined%
\gdef\SetFigFont#1#2#3#4#5{%
  \reset@font\fontsize{#1}{#2pt}%
  \fontfamily{#3}\fontseries{#4}\fontshape{#5}%
  \selectfont}%
\fi\endgroup%
\begin{picture}(11160,5445)(451,-5056)
\put(3781,-1546){\makebox(0,0)[lb]{\smash{{\SetFigFont{17}{20.4}{\familydefault}{\mddefault}{\updefault}{\color[rgb]{0,0,0}$y=\psi(u)$}%
}}}}
\put(1261,-646){\makebox(0,0)[lb]{\smash{{\SetFigFont{17}{20.4}{\familydefault}{\mddefault}{\updefault}{\color[rgb]{0,0,0}$1$}%
}}}}
\put(6796,-4606){\makebox(0,0)[lb]{\smash{{\SetFigFont{17}{20.4}{\familydefault}{\mddefault}{\updefault}{\color[rgb]{0,0,0}$\frac{5-\sqrt{17}}{4}$}%
}}}}
\put(5806,-4606){\makebox(0,0)[lb]{\smash{{\SetFigFont{17}{20.4}{\familydefault}{\mddefault}{\updefault}{\color[rgb]{0,0,0}$\frac{3-\sqrt{7}}{2}$}%
}}}}
\put(8551,-4606){\makebox(0,0)[lb]{\smash{{\SetFigFont{17}{20.4}{\familydefault}{\mddefault}{\updefault}{\color[rgb]{0,0,0}$1-\sqrt{2}/2$}%
}}}}
\put(946,-3436){\makebox(0,0)[lb]{\smash{{\SetFigFont{17}{20.4}{\familydefault}{\mddefault}{\updefault}{\color[rgb]{0,0,0}$\frac{5-\sqrt{17}}{4}$}%
}}}}
\put(856,-2941){\makebox(0,0)[lb]{\smash{{\SetFigFont{17}{20.4}{\familydefault}{\mddefault}{\updefault}{\color[rgb]{0,0,0}$3-\sqrt{7}$}%
}}}}
\end{picture}%

%% file: images/alpha-h-gamma.pdf_t
\begin{picture}(0,0)%
\includegraphics{images/alpha-h-gamma.pdf}%
\end{picture}%
\setlength{\unitlength}{4144sp}%
\begingroup\makeatletter\ifx\SetFigFont\undefined%
\gdef\SetFigFont#1#2#3#4#5{%
  \reset@font\fontsize{#1}{#2pt}%
  \fontfamily{#3}\fontseries{#4}\fontshape{#5}%
  \selectfont}%
\fi\endgroup%
\begin{picture}(11250,6660)(406,-6226)
\put(10846,-2266){\makebox(0,0)[lb]{\smash{{\SetFigFont{17}{20.4}{\familydefault}{\mddefault}{\updefault}{\color[rgb]{0,0,0}$t_0$}%
}}}}
\put(1171,-2311){\makebox(0,0)[lb]{\smash{{\SetFigFont{17}{20.4}{\familydefault}{\mddefault}{\updefault}{\color[rgb]{0,0,0}$t_1$}%
}}}}
\put(5356,-1951){\makebox(0,0)[lb]{\smash{{\SetFigFont{17}{20.4}{\familydefault}{\mddefault}{\updefault}{\color[rgb]{0,0,0}$t_2$}%
}}}}
\put(5941,-2266){\makebox(0,0)[lb]{\smash{{\SetFigFont{17}{20.4}{\familydefault}{\mddefault}{\updefault}{\color[rgb]{0,0,0}$t_3$}%
}}}}
\end{picture}%

%% file: images/alpha-gamma.pdf_t
\begin{picture}(0,0)%
\includegraphics{images/alpha-gamma.pdf}%
\end{picture}%
\setlength{\unitlength}{4144sp}%
\begingroup\makeatletter\ifx\SetFigFont\undefined%
\gdef\SetFigFont#1#2#3#4#5{%
  \reset@font\fontsize{#1}{#2pt}%
  \fontfamily{#3}\fontseries{#4}\fontshape{#5}%
  \selectfont}%
\fi\endgroup%
\begin{picture}(11580,6441)(211,-6052)
\put(271,-4381){\makebox(0,0)[lb]{\smash{{\SetFigFont{17}{20.4}{\familydefault}{\mddefault}{\updefault}{\color[rgb]{0,0,0}$2^{15}$}%
}}}}
\put(9091,-5956){\makebox(0,0)[lb]{\smash{{\SetFigFont{17}{20.4}{\familydefault}{\mddefault}{\updefault}{\color[rgb]{0,0,0}$\frac{3-\sqrt{7}}{2}$}%
}}}}
\put(226,-16){\makebox(0,0)[lb]{\smash{{\SetFigFont{17}{20.4}{\familydefault}{\mddefault}{\updefault}{\color[rgb]{0,0,0}$2^ {63}$}%
}}}}
\put(271,-2941){\makebox(0,0)[lb]{\smash{{\SetFigFont{17}{20.4}{\familydefault}{\mddefault}{\updefault}{\color[rgb]{0,0,0}$2^{31}$}%
}}}}
\put(271,-5146){\makebox(0,0)[lb]{\smash{{\SetFigFont{17}{20.4}{\familydefault}{\mddefault}{\updefault}{\color[rgb]{0,0,0}$2^ 7$}%
}}}}
\put(271,-5506){\makebox(0,0)[lb]{\smash{{\SetFigFont{17}{20.4}{\familydefault}{\mddefault}{\updefault}{\color[rgb]{0,0,0}$2^ 3$}%
}}}}
\put(271,-5686){\makebox(0,0)[lb]{\smash{{\SetFigFont{17}{20.4}{\familydefault}{\mddefault}{\updefault}{\color[rgb]{0,0,0}$2$}%
}}}}
\put(11206,-5956){\makebox(0,0)[lb]{\smash{{\SetFigFont{17}{20.4}{\familydefault}{\mddefault}{\updefault}{\color[rgb]{0,0,0}$\frac{5-\sqrt{17}}{4}$}%
}}}}
\put(2566,-5146){\makebox(0,0)[lb]{\smash{{\SetFigFont{17}{20.4}{\familydefault}{\mddefault}{\updefault}{\color[rgb]{0,0,0}$i=1$}%
}}}}
\put(4051,-4516){\makebox(0,0)[lb]{\smash{{\SetFigFont{17}{20.4}{\familydefault}{\mddefault}{\updefault}{\color[rgb]{0,0,0}$i=2$}%
}}}}
\put(5221,-3571){\makebox(0,0)[lb]{\smash{{\SetFigFont{17}{20.4}{\familydefault}{\mddefault}{\updefault}{\color[rgb]{0,0,0}$i=3$}%
}}}}
\put(7246,-2311){\makebox(0,0)[lb]{\smash{{\SetFigFont{17}{20.4}{\familydefault}{\mddefault}{\updefault}{\color[rgb]{0,0,0}$i=4$}%
}}}}
\end{picture}%

%% file: images/alpha-h-beta-gamma.pdf_t
\begin{picture}(0,0)%
\includegraphics{images/alpha-h-beta-gamma.pdf}%
\end{picture}%
\setlength{\unitlength}{4144sp}%
\begingroup\makeatletter\ifx\SetFigFont\undefined%
\gdef\SetFigFont#1#2#3#4#5{%
  \reset@font\fontsize{#1}{#2pt}%
  \fontfamily{#3}\fontseries{#4}\fontshape{#5}%
  \selectfont}%
\fi\endgroup%
\begin{picture}(11250,6885)(406,-6586)
\put(1531,-5101){\makebox(0,0)[lb]{\smash{{\SetFigFont{17}{20.4}{\familydefault}{\mddefault}{\updefault}{\color[rgb]{0,0,0}$t_0=0$}%
}}}}
\put(3691,-5146){\makebox(0,0)[lb]{\smash{{\SetFigFont{17}{20.4}{\familydefault}{\mddefault}{\updefault}{\color[rgb]{0,0,0}$t_1$}%
}}}}
\put(4186,-4561){\makebox(0,0)[lb]{\smash{{\SetFigFont{17}{20.4}{\familydefault}{\mddefault}{\updefault}{\color[rgb]{0,0,0}$t_2$}%
}}}}
\put(4546,-5101){\makebox(0,0)[lb]{\smash{{\SetFigFont{17}{20.4}{\familydefault}{\mddefault}{\updefault}{\color[rgb]{0,0,0}$\zeta_1$}%
}}}}
\put(9766,-5101){\makebox(0,0)[lb]{\smash{{\SetFigFont{17}{20.4}{\familydefault}{\mddefault}{\updefault}{\color[rgb]{0,0,0}$\zeta_2$}%
}}}}
\end{picture}%

%% file: images/alpha-alpha.pdf_t
\begin{picture}(0,0)%
\includegraphics{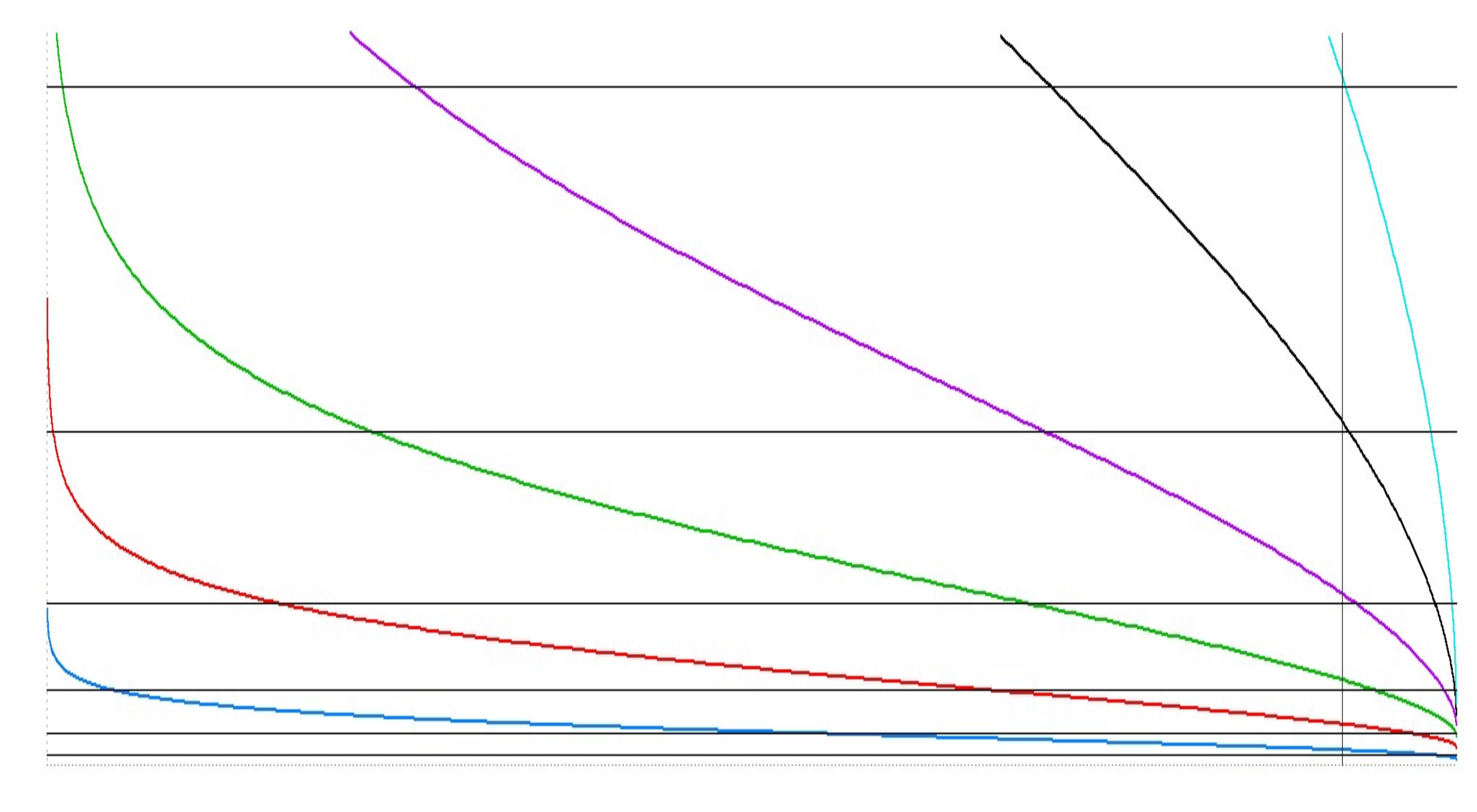}%
\end{picture}%
\setlength{\unitlength}{4144sp}%
\begingroup\makeatletter\ifx\SetFigFont\undefined%
\gdef\SetFigFont#1#2#3#4#5{%
  \reset@font\fontsize{#1}{#2pt}%
  \fontfamily{#3}\fontseries{#4}\fontshape{#5}%
  \selectfont}%
\fi\endgroup%
\begin{picture}(11625,6396)(256,-6052)
\put(316,-5686){\makebox(0,0)[lb]{\smash{{\SetFigFont{17}{20.4}{\familydefault}{\mddefault}{\updefault}{\color[rgb]{0,0,0}$2$}%
}}}}
\put(316,-5461){\makebox(0,0)[lb]{\smash{{\SetFigFont{17}{20.4}{\familydefault}{\mddefault}{\updefault}{\color[rgb]{0,0,0}$2^3$}%
}}}}
\put(316,-5146){\makebox(0,0)[lb]{\smash{{\SetFigFont{17}{20.4}{\familydefault}{\mddefault}{\updefault}{\color[rgb]{0,0,0}$2^7$}%
}}}}
\put(361,-4516){\makebox(0,0)[lb]{\smash{{\SetFigFont{17}{20.4}{\familydefault}{\mddefault}{\updefault}{\color[rgb]{0,0,0}$2^{15}$}%
}}}}
\put(361,-3076){\makebox(0,0)[lb]{\smash{{\SetFigFont{17}{20.4}{\familydefault}{\mddefault}{\updefault}{\color[rgb]{0,0,0}$2^{31}$}%
}}}}
\put(271,-376){\makebox(0,0)[lb]{\smash{{\SetFigFont{17}{20.4}{\familydefault}{\mddefault}{\updefault}{\color[rgb]{0,0,0}$2^{63}$}%
}}}}
\put(11251,-5956){\makebox(0,0)[lb]{\smash{{\SetFigFont{17}{20.4}{\familydefault}{\mddefault}{\updefault}{\color[rgb]{0,0,0}$2-3\sqrt{2}$}%
}}}}
\put(11071,-826){\makebox(0,0)[lb]{\smash{{\SetFigFont{17}{20.4}{\familydefault}{\mddefault}{\updefault}{\color[rgb]{0,0,0}$i=6$}%
}}}}
\put(10396,-5956){\makebox(0,0)[lb]{\smash{{\SetFigFont{17}{20.4}{\familydefault}{\mddefault}{\updefault}{\color[rgb]{0,0,0}$\frac{13-3\sqrt{17}}{4}$}%
}}}}
\put(3511,-5236){\makebox(0,0)[lb]{\smash{{\SetFigFont{17}{20.4}{\familydefault}{\mddefault}{\updefault}{\color[rgb]{0,0,0}$i=1$}%
}}}}
\put(4321,-4606){\makebox(0,0)[lb]{\smash{{\SetFigFont{17}{20.4}{\familydefault}{\mddefault}{\updefault}{\color[rgb]{0,0,0}$i=2$}%
}}}}
\put(5311,-3616){\makebox(0,0)[lb]{\smash{{\SetFigFont{17}{20.4}{\familydefault}{\mddefault}{\updefault}{\color[rgb]{0,0,0}$i=3$}%
}}}}
\put(6976,-2311){\makebox(0,0)[lb]{\smash{{\SetFigFont{17}{20.4}{\familydefault}{\mddefault}{\updefault}{\color[rgb]{0,0,0}$i=4$}%
}}}}
\put(9406,-1186){\makebox(0,0)[lb]{\smash{{\SetFigFont{17}{20.4}{\familydefault}{\mddefault}{\updefault}{\color[rgb]{0,0,0}$i=5$}%
}}}}
\end{picture}%

%% file: uimp.bbl
\begin{bibsection}
\begin{biblist}



\bib{BCSS}{book}{
   author={Blum, Lenore},
   author={Cucker, Felipe},
   author={Shub, Michael},
   author={Smale, Steve},
   title={Complexity and real computation},
   note={With a foreword by Richard M. Karp},
   publisher={Springer-Verlag},
   place={New York},
   date={1998},
   pages={xvi+453},
   isbn={0-387-98281-7},
   review={\MR{1479636 (99a:68070)}},
}
\smallskip

\bib{BC06}{article}{
   author={B{\"u}rgisser, Peter},
   author={Cucker, Felipe},
   title={Counting complexity classes for numeric computations. II.
   Algebraic and semialgebraic sets},
   journal={J. Complexity},
   volume={22},
   date={2006},
   number={2},
   pages={147--191},
   issn={0885-064X},
   review={\MR{2200367 (2007b:68059)}},
   doi={10.1016/j.jco.2005.11.001},
}
\smallskip

\bib{Buergisser-Cucker}{article}{
   author={B{\"u}rgisser, Peter},
   author={Cucker, Felipe},
   title={On a problem posed by Steve Smale},
   journal={Ann. of Math. (2)},
   volume={174},
   date={2011},
   number={3},
   pages={1785--1836},
   issn={0003-486X},
   review={\MR{2846491}},
   doi={10.4007/annals.2011.174.3.8},
}
		



\smallskip


\bib{BCL}{article}{
   author={B{\"u}rgisser, Peter},
   author={Cucker, Felipe},
   author={Lotz, Martin},
   title={The probability that a slightly perturbed numerical analysis
   problem is difficult},
   journal={Math. Comp.},
   volume={77},
   date={2008},
   number={263},
   pages={1559--1583},
   issn={0025-5718},
   review={\MR{2398780 (2009a:65132)}},
   doi={10.1090/S0025-5718-08-02060-7},
}
	
\smallskip

\bib{CKMW1}{article}{
author={Cucker, Felipe},
author={Krick, Teresa},
author={Malajovich, Gregorio},
author={Wschebor, Mario},
title={A numerical algorithm for zero counting I: Complexity and accuracy},
journal={Journal of Complexity},
month={Oct-Dec},
year={2008},
volume={24},
number={5-6},
pages={582-605}
doi={10.1016/j.jco.2008.03.001},
}
\smallskip

\bib{CKMW2}{article}{
author={Cucker, Felipe},
author={Krick, Teresa},
author={Malajovich, Gregorio},
author={Wschebor, Mario},
title={A numerical algorithm for zero counting II: Distance to Ill-posedness
and Smoothed Analysis},
journal={Journal of Fixed Point Theory and Applications},
volume={6},
number={2},
pages={285-294},
month={dec},
year={2009},
doi={10.1007/s11784-009-0127-4}
}
\smallskip


\bib{CKMW3}{article}{
   author={Cucker, Felipe},
   author={Krick, Teresa},
   author={Malajovich, Gregorio},
   author={Wschebor, Mario},
   title={A numerical algorithm for zero counting. III: Randomization and
   condition},
   journal={Adv. in Appl. Math.},
   volume={48},
   date={2012},
   number={1},
   pages={215--248},
   issn={0196-8858},
   review={\MR{2845516}},
   doi={10.1016/j.aam.2011.07.001},
}
		


\smallskip

\bib{Dedieu-separation}{article}{
   author={Dedieu, Jean-Pierre},
   title={Estimations for the separation number of a polynomial system},
   journal={J. Symbolic Comput.},
   volume={24},
   date={1997},
   number={6},
   pages={683--693},
   issn={0747-7171},
   review={\MR{1487794 (99b:65065)}},
   doi={10.1006/jsco.1997.0161},
}
\smallskip

\bib{Dedieu-points-fixes}{article}{
   author={Dedieu, Jean-Pierre},
   title={Estimations for the separation number of a polynomial system},
   journal={J. Symbolic Comput.},
   volume={24},
   date={1997},
   number={6},
   pages={683--693},
   issn={0747-7171},
   review={\MR{1487794 (99b:65065)}},
   doi={10.1006/jsco.1997.0161},
}
\smallskip

\bib{DMS}{article}{ author={Dedieu, Jean-Pierre}, author={Gregorio Malajovich}, author={Mike Shub}, title={Adaptive Step Size Selection for Homotopy Methods to Solve Polynomial Equations }, journal={IMA Journal of Numerical Analysis} year={2012}, note={\url{http://dx.doi.org/doi:10.1093/imanum/drs007}  
},}


\smallskip
\bib{DEMMEL-CONDITION}{article}{
   author={Demmel, James W.},
   title={The probability that a numerical analysis problem is difficult},
   journal={Math. Comp.},
   volume={50},
   date={1988},
   number={182},
   pages={449--480},
   issn={0025-5718},
   review={\MR{929546 (89g:65062)}},
   doi={10.2307/2008617},
}
\smallskip
\bib{EDELMAN}{article}{
   author={Edelman, Alan},
   title={On the distribution of a scaled condition number},
   journal={Math. Comp.},
   volume={58},
   date={1992},
   number={197},
   pages={185--190},
   issn={0025-5718},
   review={\MR{1106966 (92g:15034)}},
   doi={10.2307/2153027},
}
	

\smallskip
\bib{Eckart-Young2}{article}{
   author={Eckart, Carl},
   author={Young, Gale},
   title={The approximation of a matrix by another of lower rank},
   journal={Psychometrika},
   volume={1},
   date={1936},
   number={3},
   pages={211--218},
   doi={10.1007/BF02288367},
}
\smallskip
\bib{Eckart-Young}{article}{
   author={Eckart, Carl},
   author={Young, Gale},
   title={A principal axis transformation for non-hermitian matrices},
   journal={Bull. Amer. Math. Soc.},
   volume={45},
   date={1939},
   number={2},
   pages={118--121},
   doi={10.1090/S0002-9904-1939-06910-3},
}
\smallskip
\bib{GLSY}{article}{
   author={Giusti, M.},
   author={Lecerf, G.},
   author={Salvy, B.},
   author={Yakoubsohn, J.-C.},
   title={On location and approximation of clusters of zeros: case of
   embedding dimension one},
   journal={Found. Comput. Math.},
   volume={7},
   date={2007},
   number={1},
   pages={1--49},
   issn={1615-3375},
   review={\MR{2283341 (2008e:65159)}},
   doi={10.1007/s10208-004-0159-5},
}

\smallskip
\bib{Higham}{book}{
   author={Higham, Nicholas J.},
   title={Accuracy and stability of numerical algorithms},
   edition={2},
   publisher={Society for Industrial and Applied Mathematics (SIAM)},
   place={Philadelphia, PA},
   date={2002},
   pages={xxx+680},
   isbn={0-89871-521-0},
   review={\MR{1927606 (2003g:65064)}},
}
\smallskip


\bib{IEEE-754}{report}{
   author={The Institute of Electrical and Electronics Engineers Inc},
   title={IEEE Standard for Floating Point Arithmetic
   IEEE Std 754-2008},
   eprint={http://ieeexplore.ieee.org/xpl/standards.jsp},
   address={3 Park Avenue, New York, NY 10016-5997, USA},
   date={2008}
}
\smallskip


\bib{Kantorovich}{article}{
   author={Kantorovich, L. V.},
   title={On the Newton method},
   note={Article originally published in {\em Trudy MIAN SSSR} {\bf 28} 104-144(1949).},
   book={
     title={in: L.V. Kantorovich, {\em Selected works. Part II,
     Applied functional analysis. Approximation methods and computers;}},
     series={Classics of Soviet Mathematics},
     volume={3},
     note={
   Translated from the Russian by A. B. Sossinskii;
   Edited by S. S. Kutateladze and J. V. Romanovsky},
   publisher={Gordon and Breach Publishers, Amsterdam},
     date={1996},
     pages={ii+393},
     isbn={2-88449-013-2},
     review={\MR{1800892 (2002e:01067)}}},
   year={1949}
}
\smallskip


\bib{Malajovich-PhD}{thesis}{
   author={Malajovich, Gregorio},
   title={On the complexity of path-following Newton algorithms for solving systems of polynomial equations with integer coefficients},
   type={PhD Thesis},
   organization={Department of Mathematics,
   University of California at Berkeley},
   eprint={http://www.labma.ufrj.br/~gregorio/papers/thesis.pdf},
   year={1993}
   }
\smallskip


\bib{Malajovich94}{article}{
   author={Malajovich, Gregorio},
   title={On generalized Newton algorithms: quadratic convergence,
   path-following and error analysis},
   note={Selected papers of the Workshop on Continuous Algorithms and
   Complexity (Barcelona, 1993)},
   journal={Theoret. Comput. Sci.},
   volume={133},
   date={1994},
   number={1},
   pages={65--84},
   issn={0304-3975},
   review={\MR{1294426 (95g:65073)}},
   doi={10.1016/0304-3975(94)00065-4},
}
\smallskip


\bib{NONLINEAR-EQUATIONS}{book}{
   author={Malajovich, Gregorio},
   title={Nonlinear equations},
   series={Publica\c c\~oes Matem\'aticas do IMPA. [IMPA Mathematical
   Publications]},
   note={With an appendix by Carlos Beltr\'an, Jean-Pierre Dedieu, Luis
   Miguel Pardo and Mike Shub;
   28$^{\rm o}$ Col\'oquio Brasileiro de Matem\'atica. [28th Brazilian
   Mathematics Colloquium]},
   publisher={Instituto Nacional de Matem\'atica Pura e Aplicada (IMPA), Rio
   de Janeiro},
   date={2011},
   pages={xiv+177},
   isbn={978-85-244-0329-3},
   review={\MR{2798351 (2012j:65148)}},
}
\smallskip

\bib{MEER2000}{article}{ author={Meer, Klaus}, title={Counting problems over the reals}, journal={Theoret. Comput. Sci.}, volume={242}, date={2000}, number={1-2}, pages={41--58}, issn={0304-3975}, review={\MR{1769145 (2002g:68041)}}, doi={10.1016/S0304-3975(98)00190-X}, }
\smallskip


\bib{Nachbin0}{book}{
   author={Nachbin, Leopoldo},
   title={Lectures on the Theory of Distributions},
   series={Textos de Matemática},
   publisher={Instituto de Física e Matemática, Universidade do Recife},
   date={1964},
   pages={279},
}
\smallskip

\bib{Nachbin}{book}{
   author={Nachbin, Leopoldo},
   title={Topology on spaces of holomorphic mappings},
   series={Ergebnisse der Mathematik und ihrer Grenzgebiete, Band 47},
   publisher={Springer-Verlag New York Inc., New York},
   date={1969},
   pages={v+66},
   review={\MR{0254579 (40 \#7787)}},
}
\smallskip
\bib{Bezout1}{article}{
   author={Shub, Michael},
   author={Smale, Steve},
   title={Complexity of B\'ezout's theorem. I. Geometric aspects},
   journal={J. Amer. Math. Soc.},
   volume={6},
   date={1993},
   number={2},
   pages={459--501},
   issn={0894-0347},
   review={\MR{1175980 (93k:65045)}},
   doi={10.2307/2152805},
}

\smallskip
\bib{Smale-analysis}{article}{
   author={Smale, Steve},
   title={On the efficiency of algorithms of analysis},
   journal={Bull. Amer. Math. Soc. (N.S.)},
   volume={13},
   date={1985},
   number={2},
   pages={87--121},
   issn={0273-0979},
   review={\MR{799791 (86m:65061)}},
   doi={10.1090/S0273-0979-1985-15391-1},
}
\smallskip


\bib{Smale-PE}{article}{
   author={Smale, Steve},
   title={Newton's method estimates from data at one point},
   conference={
      title={ computational mathematics},
      address={Laramie, Wyo.},
      date={1985},
   },
   book={
      publisher={Springer},
      place={New York},
   },
   date={1986},
   pages={185--196},
   review={\MR{870648 (88e:65076)}},
}
\smallskip
\bib{Spielman-Teng}{article}{
   author={Spielman, Daniel A.},
   author={Teng, Shang-Hua},
   title={Smoothed analysis of algorithms: why the simplex algorithm usually
   takes polynomial time},
   journal={J. ACM},
   volume={51},
   date={2004},
   number={3},
   pages={385--463 (electronic)},
   issn={0004-5411},
   review={\MR{2145860 (2006f:90029)}},
   doi={10.1145/990308.990310},
}

\smallskip
\bib{Turing}{article}{
   author={Turing, A. M.},
   title={Rounding-off errors in matrix processes},
   journal={Quart. J. Mech. Appl. Math.},
   volume={1},
   date={1948},
   pages={287--308},
   issn={0033-5614},
   review={\MR{0028100 (10,405c)}},
}
\smallskip

\bib{Xinghua}{article}{
author={Wang Xinghua},
title={Some result relevant to Smale's reports},
book={title={in: M.Hirsch, J. Marsden and S. Shub(eds):
  {\em From Topolgy to Computation: Proceedings of Smalefest}},
publisher={Springer},
place={new-york}},
year={1993},
pages={456-465}}

\smallskip

\bib{Wilkinson}{book}{
   author={Wilkinson, J. H.},
   title={Rounding errors in algebraic processes},
   note={Reprint of the 1963 original [Prentice-Hall, Englewood Cliffs, NJ;
   MR0161456 (28 \#4661)]},
   publisher={Dover Publications Inc.},
   place={New York},
   date={1994},
   pages={viii+161},
   isbn={0-486-67999-3},
   review={\MR{1280465}},
}


\end{biblist}
\end{bibsection}
